\documentclass[12pt]{amsart}

\usepackage{amssymb,amsthm,amsmath,txfonts,a4wide}
\usepackage{xcolor}
\usepackage{enumerate}

\usepackage{hyperref}
\usepackage{tikz}
\usetikzlibrary{shapes,arrows}
\usepackage{times}
\usepackage{verbatim}

\hypersetup{
	colorlinks,
	citecolor=green,
	filecolor=blue,
	linkcolor=red,
	urlcolor=blue
}

\usepackage[greek,english]{babel}

\newtheorem{defn}{Definition}[section]
\newtheorem{lemma}[defn]{Lemma}
\newtheorem{remark}[defn]{Remark}

\newtheorem{theorem}[defn]{Theorem}
\newtheorem{corollary}[defn]{Corollary}

\newtheorem{definition}[defn]{Definition}
\newtheorem{proposition}[defn]{Proposition}

\newtheorem{claim}[defn]{Claim}

\numberwithin{equation}{section}

\begin{document}
	\title{Ill-posedness of quintic fourth order Schr\"odinger equation}
	\author{Bo Xia}
	\address{School of Mathematical Sciences, USTC, Hefei, P. R. China}
	\email{xiabomath@ustc.edu.cn, xaboustc@hotmail.com}
	\thanks{}
	\author{Deng Zhang} 
	\address{School of Mathematical Sciences, CMA-Shanghai, Shanghai Jiao Tong University, China.}
	\email{dzhang@sjtu.edu.cn}
	\thanks{}
	
\subjclass[2010]{}



	\date{}
	\maketitle 
	
	\begin{abstract}
		We prove that the solution map, associated to the quintic fourth order nonlinear 
		Schr\"odinger equation, exhibits the norm inflation phenomenon at every point in the Sobolev spaces of super-critical regularity. Indeed, we prove this result separately in the cases of negative and of positive regularity. 
		In the negative regularity case, we prove the result for both the defocusing and focusing equations: in the one dimensional case, the associated solution map exhibits norm inflation in Sobolev spaces of super-critical regularity (including the critical index); in the higher dimensional case, the solution map exhibits the same phenomenon in spaces of negative regularity. Meanwhile, in the case of positive regularity, we prove the result for the defocusing equation in dimensions $d=3,4,5$. 
		Our proofs are based on the 
		``high-to-low'' and ``low-to-high'' frequency cascades respectively. 
	\end{abstract}
	
	\maketitle

	\section{Introduction}
	In this article, we will consider the ill-posedness problem 
	for the quintic fourth order Schr\"{o}dinger equation
	\begin{equation}\label{eq:main}
		\left\{
		\begin{split}
			& i\partial_tu+\Delta^2 u+ \mu|u|^4u=0, \\
			& u|_{t=0}=u_0,
		\end{split}
		\right.
	\end{equation}
	where  $\mu =1\ \mathrm{or}\ -1$, and  $u:\mathbb{R}^d\times\mathbb{R}\ni(x,t)\mapsto u(x,t)\in\mathbb{C}$ is the unknown. 
	
	Historically, the fourth order Schr\"{o}dinger equation was proposed to study the propagation of intense laser beams in magnetic materials by taking into account higher order dispersion terms, see   Karpman \cite{MR1372681} and Karpman-Shagalov \cite{MR1779828}. Indeed they proposed the following more general equation
	\begin{equation}\label{eq:4nls}
		i\partial_tu+ \Delta^2 u+\beta \Delta u+ \mu|u|^{2p}u=0,
	\end{equation}
	where $\beta$ and $\mu$ are real numbers, and $p$ is a positive number.  The mathematical study of \eqref{eq:4nls} was started in \cite{benartzi2000} where {Ben-Artizi-Koch-Saut} proved the dispersion estimates for the linear operator in \eqref{eq:4nls}, and in \cite{fibichilanpapanicolaou} where Fibich-Ilan-Papanicolaou analysed the self-focusing and singularity formation for \eqref{eq:4nls}. 
	More recently, Cui \cite{cui2005,cui2006} and Pausader \cite{pausader2007} obtained the Strichartz estimates,
	based on which Guo-Cui \cite{guocui2006,guocui2007b}, Miao-Xu-Zhao \cite{miaoxuzhao2009,miaoxuzhao2011}, and Pausader \cite{pausader2007} 
	considered the Cauchy problem of \eqref{eq:4nls} in the energy space,
	for certain ranges of $\beta,\mu$ and $p$.
	The Cauchy problem for \eqref{eq:4nls} in the low regularity space $H^s$ was also studied, see \cite{guo2010} and \cite{liuzhang2021}.
	
	Although \eqref{eq:4nls} resembles a lot of the classical nonlinear Schr\"{o}dinger equation, several basic questions have not been considered yet. Here, we will restrict ourselves to the simple model \eqref{eq:main} and consider the ill-posedness of \eqref{eq:main} in spaces of low regularity.
	
	We first recall some facts about \eqref{eq:main}. 
	Equation \eqref{eq:main} has two types of symmetries.
	The first one is
	related to the invariance of phase rotation and time translation,
	which correspond respectively,
	to the conservation laws of the mass
	\begin{equation}\label{eq:mass}
		\mathcal{M}[u](t):=\int |u|^2dx
	\end{equation}
	and the energy
	\begin{equation}\label{eq:energy}
		\mathcal{H}[u](t):=\int \left[\frac{1}{2}\left|\Delta u\right|^2 +\frac{\mu}{6}|u|^{6} \right]dx.
	\end{equation}
	The second important symmetry is the scaling invariance:
	if $u$ is a solution to \eqref{eq:main},
	then so is the rescaled function $u_{\lambda}$, defined by
	\begin{equation*}
		u_{\lambda}(x,t):=\frac{1}{\lambda}u\left(\frac{x}{\lambda},\frac{t}{\lambda^4}\right),\ \ \lambda >0.
	\end{equation*}
	In particular, one has
	\begin{equation*}
		\left\|u_\lambda(\cdot,0)\right\|_{\dot{H}^{s_{cr}(d)}}=\left\|u(\cdot,0)\right\|_{\dot{H}^{s_{cr}(d)}},
	\end{equation*}
	where
	\begin{equation} \label{scr-def}
		s_{cr}(d):=\frac{d}{2}-1.
	\end{equation}
	Here $\dot{H}^s$ denotes the homogeneous $L^2$-based Sobolev space of regularity $s$, for the given real number $s$.
	
	It is generally expected that  $s_{cr}(d)$ serves as the critical index
	for the well-posedness of equation \eqref{eq:main}.
	In the case $s_{cr}(d)\geq0$, on the one hand, it is believed that \eqref{eq:main} is well-posed in $H^s$ for $s\geq s_{cr}(d)$, {see \cite{dinh2018,dinh2018b,miaoxuzhao2009,miaoxuzhao2011,pausader2009b}} (see also the recent result by the first author \cite{xia} for \eqref{eq:main} on the circle);
	on the other hand,
	it is believed to be ill-posed in $H^s$ for $s<s_{cr}(d)$. However, in the case $s_{cr}(d)<0$, we can not always insist on such a belief, see the forthcoming Remark \ref{rmk:2} for more details.
	
	We next recall three different notions of the ill-posedness  in the literature.
	
	\begin{definition}\label{def}
		We adapt notions of ill-posedness in \cite{christcolliandertao2003}, \cite{kenigponcevega2001} and {\cite{oh2017}} for canonical equations, to \eqref{eq:main} as follows.
		\begin{itemize}
			\item We say that the equation \eqref{eq:main} does not have unique solutions,
			if there are two different solutions $u_1$ and $u_2$, such that $u_1(0)=u_2(0)$;
			\item The solution map from $H^s$ to $C([-T,T],H^s)$, $T>0$,
			is said to be discontinuous,
			{if for any $\epsilon>0$, there exist a solution $u$ to \eqref{eq:main}  and 
				$t\in [-T,T]\cap (0,\epsilon)$,
				such that }
			\begin{equation}
				\|u(\cdot,t)-u(\cdot, 0)\|_{H^s}<\epsilon,\ \ \mathit{and}\ \ \|u(\cdot,t)\|_{H^s}\gtrsim 1;
			\end{equation}
			\item We say that equation \eqref{eq:main} exhibits the norm inflation phenomenon at $u_0\in H^s$
			if for any $\epsilon>0$, there exist a solution $u$ to \eqref{eq:main} and $t\in  (0,\epsilon)$ such that
			\begin{equation}
				\|u(\cdot,0)-u_0\|_{H^s}<\epsilon,\ \ \mathit{and}\ \ \|u(\cdot, t)\|_{H^s}>\epsilon^{-1};
			\end{equation}
		\end{itemize}
	\end{definition}

	In \cite{kenigponcevega2001},
	Kenig-Ponce-Vega {proved the discontinuity of the solution map associated to the cubic NLS, KdV and mKdV, 
		and they also showed that the 1D cubic NLS cannot have 
		a unique weak solution starting from the delta function if there exists any}. 
	
	For wave equations,
	Lindblad \cite{lindblad1993} proved that the solutions to the 3D wave equation
	exhibit the concentration phenomenon in the physical space,
	and hence the corresponding solution map exhibits the norm inflation phenomenon.
	Afterwards, Lindblad-Sogge \cite{lindbladsogge1995} proved that the 3D wave equation
	is locally well-posed in the subcritical regime,
	while the solution map is not continuous in the supercritical regime.
	Furthermore,  for the 3D wave equation,
	Lebeau \cite{lebeau1999} proved that,
	in the energy super-critical regime,
	the solution map fails to be uniformly continuous in the energy space,
	by invoking the methods from geometric optics.
	
	The norm inflation phenomenon
	at $0$ in $H^s$ for some $s<0$
	was proved by Christ-Colliander-Tao \cite{christcolliandertao2003}
	for a number of canonical equations,
	{including KdV equations} and Schr\"odinger equations.
	The proof in \cite{christcolliandertao2003} is based on the frequency modulation method,
	and has been used to prove the ill-posedness in low regularity spaces
	for more general Hamiltonian equations.
	By capturing the high-to-low frequency cascade in the first Picard iteration,
	Bejenaru-Tao \cite{bejenarutao2006} proved that the solution map of 1D quadratic Schr\"{o}dinger equation
	fails to be continuous from $H^s(\mathbb{R})$ to $C([-T,T],H^{-1}(\mathbb{R})), T>0$, for any $s<-1$.
	Bejenaru-Tao's method was further developed to the method of series expansion in \cite{iwabuchiogawa2014},
	in which Iwabuchi and Ogawa showed that the 2D quadratic Schr\"{o}dinger equation
	exhibits the norm inflation 
	{in $H^s(\mathbb{R}^2)$ with $s\leq-1$}.
	Moreover, Kishimoto \cite{kishimoto2019} showed that the norm  inflation phenomenon
	also occur for the higher dimensional Schr\"{o}digner equations spaces of low regularity.
	See also the work by Choffrut and Pocuvnicu  \cite{choffrutpocovnicu2018}
	for the cubic fractional Schr\"{o}dinger equations, that by Seong \cite{seong2021} for the cubic fourth order Schr\"{o}dinger equation,
	and that by Dinh \cite{dinh2018b} for the fourth order Schr\"{o}dinger equations with more general nonlinearities.
	
	It is worth noting that,
	the norm inflation phenomena in the aforementioned works exhibit at the zero data
	in Sobolev spaces.
	Such phenomena, actually,
	can happen at any data.
	In \cite{oh2017}, Oh affirmed this for the cubic Schr\"{o}dinger equation,
	in the Sobolev spaces below the critical regularity  (including the criticality)
	in the 1D case and in the negative Sobolev spaces in higher dimensions.
	Forlano and Okamoto \cite{forlanookamoto2020} proved similar results in the negative Sobolev spaces
	for semilinear wave equations in all dimensions.
	The norm inflation was also proved by the first author \cite{xia2021}
	for the  3D wave equation
	on torus at any data in Sobolev spaces of certain positive regularity.
	The method in \cite{xia2021} is indeed an ODE-approach, based on the work of Burq-Tzvetkov \cite{burqtzvetkov2008}.
	See also a recent work \cite{suntzvetkov2020} by {Sun-Tzvetkov, concerning the pathological subset in Sobolev space of super-critical regularity.}
	
	In this paper, we aim at proving that the norm inflation phenomenon occurs for \eqref{eq:main}, at any data in $H^s(\mathbb{R}^d)$ {for certain $s$ to be specified.} We state our first result as follows.
	\begin{theorem}\label{thm:main}
		Let $d\geq 1$ be an integer and $s$ a real number such that $s\leq -\frac{1}{2}$ if $d=1$, and   $s<0$ if $d\geq 2$. 
		Then, the equation \eqref{eq:main} exhibits the norm inflation phenomenon at any  $u_0\in H^s$.
	\end{theorem}
	
	Notice that the norm inflation occurs for \eqref{eq:main} in the space of negative regularity. In the space of positive regularity (under the assumption $s_{cr}>0$), this phenomenon also occurs for \eqref{eq:main} in the defocusing case.
	\begin{theorem}\label{thm:main22}
		Let $3\leq d\leq 5$ be an integer, and $s\in (0,s_{cr}(d))$.  If $\mu=1$, then the conclusion in Theorem \ref{thm:main} also holds.
	\end{theorem}
	\begin{remark}
		We see this result has the dimension constraint $3\leq d\leq 5$ and the defocusing requirement. The reason relies on that, to construct the desired solutions of \eqref{eq:main}, we not only used the local existence of solutions in energy space for \eqref{eq:main}, but also utilized that the local solutions should live for a sufficiently long time. Both of these are now known to be true only for the above dimensions in the defocusing case, see \cite{dinh2018b} and \cite{ghanmisaanouni2016}. At this point, we further remark that for $3\leq d\leq 8$, as long as \eqref{eq:main}, starting from any smooth data, has regular solutions that live for a sufficiently long time, the solution map of \eqref{eq:main} will exhibit norm inflation everywhere (as in Theorem \ref{thm:main22}).
	\end{remark}
	
	An immediate consequence of Theorems \ref{thm:main} and \ref{thm:main22} is the {everywhere} discontinuity of the solution map associated to \eqref{eq:main}.
	\begin{corollary}\label{cor}
		Let $(d,s)$ be as in Theorem \ref{thm:main} or Theorem \ref{thm:main22}. 
		Then, for any $T>0$, the solution map $\Phi: H^s\rightarrow C([-T,T],H^s)$
		associated to the equation \eqref{eq:main} 
		{is discontinuous everywhere in $H^s$}.
	\end{corollary}
	The proof of Corollary \ref{cor} follows from Theorem \ref{thm:main} via a diagonalization argument as that in \cite{xia2021}.
	Hence, the details are omitted here.

	\begin{remark}
		The results in {Theorem \ref{thm:main} and Corollary \ref{cor} are also valid on the torus $\mathbb{T}^d$ for $s\leq -\frac{1}{2}$ if $d=1$, or $s<0$ if $d\geq2$. This can be proved by using similar arguments as in \cite{oh2017} in the torus settings.} {At this point, we further remark that, as long as \eqref{eq:main} (on $\mathbb{T}^d$ for $3\leq d\leq 8$) starting from smooth datum, has regular solutions that live for a sufficiently long time, its associated solution map will exhibit norm inflation as in Theorem \ref{thm:main22} and hence the corresponding solution map fails to be everywhere discontinuous as in Corollary \ref{cor}.} 
	\end{remark}
	
	\begin{remark}\label{rmk:1} 
		It was proved by Dihn \cite{dinh2018b} that 
		the solution map of \eqref{eq:main} from $H^s(\mathbb{R}^d)$ to $C([-T,T],H^s(\mathbb{R}^d)),T>0$
		fails to be continuous at $0$ in $H^s$,
		in the case where $s\in\left(-\infty,-\frac{d}{2}\right]\cup [0,s_{cr}(d))$ if $d\geq 3$,
		and $s\in \left(-\infty,-\frac{d}{2}\right)$ if $d=1, 2$.
		
		The proof in \cite{dinh2018b} indeed implies a stronger result:
		equation \eqref{eq:main} exhibits the norm inflation phenomenon at $0$ in $H^s(\mathbb{R}^d)$
		for $s$ in the above range except the case $s=0$.
		
		Our results in Theorems \ref{thm:main} and \ref{thm:main22} indicate that,
		the norm inflation at $0$ in $H^s$ also occurs 
		for the critical value $s=-\frac d2$
		if $d=1$, $s\in [-\frac d2, 0)$ if $d=2$,
		and $s\in (-\frac d2, 0)$ if $3\leq d\leq 5$.
		Moreover, Theorem \ref{thm:main} also shows the norm inflation phenomenon
		even at any data in $H^s(\mathbb{R}^d)$,
		with $s$ specified in Theorems \ref{thm:main} and \ref{thm:main22},
		not just at the zero data.
	\end{remark}  
	
	\begin{remark}\label{rmk:2} 
		As we have seen in Remark \ref{rmk:1}, in the one dimensional case, the range of $s$ in our results, indeed contains the point $s_{cr}(1)=-\frac{1}{2}$. This also happens in Oh's work \cite{oh2017}, dealing with the one dimensional cubic Schr\"{o}dinger equation, and in Iwabuchi-Ogawa's work \cite{iwabuchiogawa2014}, in dealing two dimensional quadratic Schr\"odinger equation. Even surprising, the norm inflation indeed occurs in some spaces of regularity above the critical index, 
		{see, for instance,  Choffrut-Pocovnicu's work \cite{choffrutpocovnicu2018} on the fractional cubic NLS on the circle.}

	\end{remark}
	
	Our proofs of Theorems \ref{thm:main} and \ref{thm:main22} are based on two different mechanisms,
	namely,
	the high-to-low frequency 
	{cascade} when $s$ is negative,
	and the low-to-high frequency {cascade} when $s$ is positive.
	
	To be precise,
	in the case $s\leq -\frac{1}{2}$ if $d=1$, 
	and $s<0$ if $d\geq 2$,
	we mainly adapt the method in \cite{forlanookamoto2020,iwabuchiogawa2014,kishimoto2019,oh2017}.
	By a density argument,
	we may assume that the initial data $u_0$ is a Schwartz function.
	The idea is to reformulate the solution $u$ to \eqref{eq:main} as a series expansion indexed by $5$-ary trees
	(see \eqref{eq:power:series} below),
	which is absolutely convergent in $C([0,T],FL^1)$ for $T\sim \|u_0\|^{-4}_{FL^1}$.
	Here $FL^1$ is the Fourier-Lebesgue space.
	
	The key fact is that,
	there exists a sequence of smooth functions $\{\phi_n\}$ (see \eqref{eq:112707} below)
	such that $u_0+\phi_n$ converges to $u_0$ in $H^s$.
	However, the first Picard iteration tends to infinity along an appropriate time sequence $\{t_n\}$,
	while the higher order Picard iterations are uniformly bounded.
	Therefore,
	this forces $\|u_n(t_n)\|_{H^s}$
	to grow to infinity as $n$ tends to $\infty$.
	The underlying mechanism for the growth of first Picard iteration
	is the transfer of energy from high to low frequencies,
	which is the content of Lemma \ref{prop:lower:bound}.
	
	In the regime where $0<s<s_{cr}(d)$ for $3\leq d\leq 5$, we adapt the method in \cite{xia2021}. As in the case of negative regularity, we assume that $u_0$ is a Schwartz function. The main oscillation here is captured by the ODE
	\begin{equation}\label{ode}
		i\partial_t v_n+ |v_n|^4 v_n=0,\ \ V(0)= \psi_n.
	\end{equation}
	Note that, equation \eqref{ode} has no dispersion term. 
	Let $u_n$ be the smooth solution to \eqref{eq:main} starting from $u_0+\psi_n$. We will choose smooth functions $\psi_n$ with compact support 
	in an appropriate way such that the following holds:
	\begin{enumerate}[(i)]
		\item $\psi_n$ converges to $0$ in $H^s(\mathbb{R}^d)$;
		\item there exists a sequence $\{t_n\}$ of positive numbers tending to zero such that, as $n$ tends to infinity,
		\begin{enumerate}[(a)]
			\item[(ii.a)] $\|v_n(t_n)\|_{H^s}$ tends to infinity,
			\item[(ii.b)] $\|v_n-u_n\|_{L^\infty([0,t_n],H^s)}$ is bonded uniformly.
		\end{enumerate}
	\end{enumerate}
	In  particular, the above facts yield the norm inflation at the given data $u_0$ in $H^s$. 
	
	We first mention that, 
	in contrast to the negative regularity case where $s<0$, 
	the underlying mechanism  in $\mathrm{(ii.a)}$ relies crucially on
	{\cite[Lemma A.3]{burqtzvetkov2008}},
	which captures quantitatively the transfer of energy from low to high frequencies. 
	In order to realize $\mathrm{(ii.b)}$, we next introduce the semi-classical energy $E_n$ 
	defined by 
	\begin{equation*}
		E_n[w_n]:=n^{s}\|w_n\|_{L^2}+n^{s-4}\|\Delta^2w_n\|_{L^2}, 
	\end{equation*}
	to measure the size of the remainder 
	$w_n:=u_n-v_n-u^L$ with $u^L$ being the free propagation of $u_0$.
	By analyzing the interactions between the free propagation,
	the ODE oscillation
	and the remainder, and using the bootstrap arguments
	we are able to show $E_n[w_n]\rightarrow0$ as $n$ tends to infinity. From this, together with the boundedness of $u_0$ in $H^s$, 
	we obtain the uniform boundedness of $\|v_n-u_n\|_{L^\infty([0,t_n],H^s({\mathbb{R}^d}))}$. 
	\begin{remark}
		Although the ``low-to-high'' frequency cascade (especially the ODE approach) has been used to study 
		the norm inflation phenomena concerning wave equations (see for instance \cite{burqtzvetkov2008},\cite{xia2021} and \cite{suntzvetkov2020}), it is the first time for us to use this idea to treat Schr\"odinger-type equations. 
	\end{remark}
	
	\vspace*{3pt}
	
	{		{We digress a bit to list some notations we are going to use throughout this article.} 
		For $1\leq p\leq \infty$, we denote by $L^p := L^p(\mathbb{R}^d)$ the space of $p$-th power integrable complex-valued functions defined on $\mathbb{R}^d$. We endow it with the norm $\|\cdot\|_{L^p}$.  For $s\in\mathbb{R}$ and $p\geq 1$, we denote by $FL^{s,p}$ the Fourier-Lebesgue space, which consists of functions $f$ satisfying
		\begin{equation*}
			\left\|f\right\|_{FL^{s,p}}
			:=\left\|\langle\xi\rangle^s\hat{f}\right\|_{L^p} 
			<\infty.
		\end{equation*}
		Here $\hat{f}$ is the Fourier transform of $f$. In the case $s=0$,
		we abbreviate $FL^p:= FL^{0,p}$.
		One result we will use frequently is that $FL^1$ is an algebra. This is indeed a direct consequence of Young's inequality. In the case $p=2$, we note that 
		$FL^{s,2}$ coincides with $H^s$, the inhomogeneous $L^2$-based Sobolev space. 
		We also denote by $\dot{H}^s$ the homogeneous Sobolev space which 
		consists of functions $f$ such that 
		\begin{align*} 
			\|f\|_{\dot{H}^s} := \| |\xi|^s \hat{f} \|_{L^2} <\infty. 
		\end{align*}
		We will conventionally use $\mathbb{N}$ to denote the set of all non-negative 
		integers and use $\mathbb{N}_+:= \mathbb{N}\setminus\{0\}$ to denote the set of all positive integers. 
		
		\vspace*{2pt}		
		We end this introduction part by briefly describing the organization of remaining part of this article: in Section \ref{Sec-s-negat}, 
		we are going to prove Theorem \ref{thm:main}, while in Section \ref{Sec-s-posit}  
		we will prove Theorem \ref{thm:main22}. 
	} \vspace*{1pt}

	\section{The cases $s\leq -\frac{1}{2}$ for $d=1$ and $s<0$ for $d\geq2$}   \label{Sec-s-negat} 
	
	In this section, we prove Theorem \ref{thm:main}. where 
	\begin{align}  \label{s<0} 
		s \leq -\frac{1}{2}\  for\ d=1,\  
		\mathrm{and}\ \ s<0\ \mathrm{for}\ d\geq2. 
	\end{align} 
	The key mechanism responsible for the norm inflation phenomenon in this case
	is the `high-to-low' frequency transfer of energy.

	\begin{definition}
		A $5$-ary tree $\mathcal{T}$ is a set  of nodes together with a partial order $\leq$, satisfying
		\begin{itemize}
			\item $\mathcal{T}$ has the unique maximal element, called the root node.
			\item Let $a,b\in\mathcal{T}$ satisfy $b\leq a$. If there another element $c\in\mathcal{T}$ such that $b\leq c$ and $c\leq a$, then $b=c$ or $c=a$. In this case, we say $a$ is the parent of $b$, and $b$ is a child of $a$; 	
			\item Every node either has no child or has exactly $5$ children.
		\end{itemize}
		The node having no child is called a terminal node. The node having children is called a non-terminal node.
	\end{definition}
	
	For  any $5$-ary tree $\mathcal{T}$,
	we denote by $\mathcal{T}^0$ the collection of its non-terminal nodes,
	and by $\mathcal{T}^\infty$ the collection of its terminal nodes.
	Given any nonnegative integer $j\in \mathbb{N}$,
	we denote by $\mathbf{T}(j)$ the collection of all $5$-ary trees with $j$ non-terminal nodes.
	
	We have the following basic properties of the $5$-ary trees.
	
	\begin{lemma}\label{lem:1127} For each $j\in\mathbb{N}$, we have
		\begin{enumerate}[(i)]
			\item the cardinality of the set $\mathbf{T}(j)$ is bounded by $C^j_0$ for some $C_0>0$;
			\item If $\mathcal{T}\in\mathbf{T}(j)$, then it has exactly $5j+1$ nodes and $4j+1$ terminal nodes.
		\end{enumerate}
	\end{lemma}
	\begin{proof}See \cite{forlanookamoto2020}  
		for $k$-ary trees, 
		where $k\geq 2$ is an integer. Indeed, 
		the statement $(i)$ can be proved by using similar arguments as in the proof of Lemma 2.3 in \cite{oh2017},
		while the proof of $(ii)$ follows from an induction argument. 
	\end{proof}

	Given any $\phi\in FL^1(\mathbb{\mathbb{R}}^d)$ and any $j\in\mathbb{N}$,
	for each $\mathcal{T}\in \mathbf{T}(j)$,
	we define the space-time distribution
	\begin{equation*}
		\Psi_{\phi}(\mathcal{T})
	\end{equation*}
	by identifying any terminal node with the linear wave $S(t)\phi$,
	where $S(t)\phi:=e^{it\Delta^2}\phi$ is the linear propagation of $\phi$,
	and then by identifying the non-terminal nodes with the Duhamel integral operator $\mathcal{I}$
	\begin{equation}\label{eq:duhamel}
		\mathcal{I}[u_1,u_2,u_3,u_4,u_5]
		:=i\mu \int_0^tS(t-\tau)\left(u_1(
		\tau)\bar{u}_2(\tau)u_3(\tau)\bar{u}_4(\tau)u_5(\tau)\right)d\tau, 
	\end{equation}
	where $u_j$, $1\leq j\leq 5$, correspond to the five children of this node. 
	{For simplicity, 
		we write $\mathcal{I}[u]:= \mathcal{I}[u,u,u,u,u]$.}  
	
	More generally, given $\mathcal{T}\in\mathbf{T}(j)$, $j\in \mathbb{N}$,
	we label its terminal nodes by $a_1,\dots,a_{4j+1}$ (in its planar graphical representation, from left to right).
	For any functions $\phi_1,\dots, \phi_{4j+1}$,
	we define the associated space-time distribution $\Psi(\mathcal{T};\phi_1,\dots,\psi_{4j+1})$
	by identifying the terminal node $a_k$ with the linear wave $S(t)\phi_k$ for $k=1,\dots, 4j+1$,
	and the non-terminal nodes with the Duhamel integral operator $\mathcal{I}$ as in \eqref{eq:duhamel}.
	In particular, one has
	\begin{align}
		\Psi_{\phi}(\mathcal{T})=\Psi(\mathcal{T};\phi,\dots,\phi).
	\end{align}
	
	Thus, we can rewrite formally the Duhamel formula of the solution to \eqref{eq:main}
	\begin{equation*}
		u(t)=S(t)u_0+\mathcal{I}[u](t), 
	\end{equation*}
	in the series expansion labeled by $5$-ary trees:
	\begin{equation}\label{eq:power:series}
		u(t)=\sum_{j=0}^\infty \sum_{\mathcal{T}\in\mathbf{T}(j)}\Psi_{\phi}(\mathcal{T})(t)
		=\sum_{j=0}^\infty \Theta_j(\phi)(t),
	\end{equation}
	{where $\phi=u_0$, and}  
	\begin{equation}\label{eq:112601}
		\Theta_j(\phi)(t):=\sum_{\mathcal{T}\in\mathbf{T}(j)}\Psi_{\phi}(\mathcal{T})(t).
	\end{equation}

	The formal series expansion \eqref{eq:power:series} indeed defines a solution to \eqref{eq:main} in $L^\infty([0,T],FL^1)$ for $T$ small enough, as is shown in the following result.
	
	\begin{proposition}\label{prop:lwp}
		For any $u_0\in FL^1$, there exists a unique solution $u \in C([0,T],FL^1)$ to \eqref{eq:main},
		where $T\sim \|u_0\|^{-4}_{FL^1}$.
		Moreover,
		$u$ has the expansion \eqref{eq:power:series}
		with $u_0$ replacing $\phi$
		and the series  is absolutely convergent in $C([0,T],FL^1)$.
	\end{proposition}
	
	The proof of Proposition \ref{prop:lwp} is based on the multilinear estimates 
	in Lemma \ref{Lem-Multiesti-Theta} 
	and {the difference estimate} in Lemma \ref{lem:11270} below.
	
	\begin{lemma} (Multilinear estimates)  \label{Lem-Multiesti-Theta}
		There exists  $C>0$, such that
		for any $\phi\in FL^1$ and any $j\in\mathbb{N}$,
		\begin{align}\label{eq:11260}
			\left\|\Theta_j(\phi)(t)\right\|_{FL^1}&\leq C^j t^j\|\phi\|^{4j+1}_{FL^1},
		\end{align}
		and for any $j\in \mathbb{N}_+$,
		\begin{align}
			\left\|\Theta_j(\phi)(t)\right\|_{FL^\infty}&\leq C^j t^j\|\phi\|^2_{L^2}\|\phi\|^{4j-1}_{FL^1}.   \label{eq:112602}
		\end{align}
	\end{lemma}
	
	\begin{proof}
		In view of Lemma \ref{lem:1127} $(i)$ and \eqref{eq:112601},
		it suffices to prove that,
		for any $\phi\in FL^1$ and $\mathcal{T}\in\mathbf{T}(j)$,  $j\in\mathbb{N}$,
		\begin{align}\label{eq:1126}
			\left\|\Psi_{\phi}(\mathcal{T})(t)\right\|_{FL^1}&\lesssim t^j\|\phi\|^{4j+1}_{FL^1},
		\end{align}
		and for any $j\in\mathbb{N}_+$,
		\begin{align}
			\left\|\Psi_{\phi}(\mathcal{T})(t)\right\|_{FL^\infty}&\lesssim  t^j\|\phi\|^2_{L^2}\|\phi\|^{4j-1}_{FL^1},  \label{eq:112603}
		\end{align}
		where the implicit constants are universal.
		
		We prove \eqref{eq:1126} by induction.
		If $j=0$, then there is only one tree $\mathcal{T}$, consisting of one node, in $\mathbf{T}(0)$,
		and $\Psi_\phi(\mathcal{T})(t)=S(t)\phi$.
		Since $S(t)$ is unitary in $FL^1$,
		\eqref{eq:1126} follows. 
		
		For general $j\geq 0$,
		assuming that \eqref{eq:1126} holds for all nonnegative 
		integers $k\leq j$,
		we shall prove that it is also valid for $k=j+1$.
		
		For this purpose, given any $\mathcal{T}\in\mathbf{T}(j+1)$, we denote the five children of the root node $a$ by $a_1,\dots,a_5$.
		Let $\mathcal{T}_l$, $1\leq l \leq 5$,
		be the subtrees of $\mathcal{T}$ such that $(i)$ the root node  of $\mathcal{T}_l$ is $a_l$; $(ii)$
		if assuming $\mathcal{T}_l\in\mathbf{T}(j_l)$ with $j_l\leq j$ for each $l\in\{1,\dots,5\}$, then $j_1+\cdots+j_5=j$.
		Notice that $\Psi_{\phi}(\mathcal{T}_l)$ satisfies estimate \eqref{eq:1126} by the induction hypothesis.
		Since
		\begin{equation*}
			\Psi_\phi(\mathcal{T})(t)
			=\mathcal{I}[\Psi_{\phi}(\mathcal{T}_1),\Psi_{\phi}(\mathcal{T}_2),
			\Psi_{\phi}(\mathcal{T}_3),\Psi_{\phi}(\mathcal{T}_4),\Psi_{\phi}(\mathcal{T}_5)](t),
		\end{equation*}
		we can use consecutively the Minkowski inequality, the unitarity of $S(t)$ on $FL^1$ and the algebraic property of $FL^1$ to compute
		\begin{align} \label{PsiphiT-FL1-esti}
			\|	\Psi_\phi(\mathcal{T})(t)\|_{FL^1}
			& \leq\int_0^t \prod\limits_{l=1}^5
			\|\Psi_\phi(\mathcal{T}_l)(\tau)\|_{FL^1} d\tau \notag \\
			& \lesssim \int_0^t\prod_{l=1}^5
			\left(\tau^{j_l}\|\phi\|^{4j_l+1}_{FL^1}\right)d\tau  \notag \\
			&\lesssim t^{j+1}\|\phi\|^{4(j+1)+1}_{FL^1},
		\end{align}
		where the implicit constants are independent of $\phi$ and $j$.
		This verifies \eqref{eq:1126} for $\mathcal{T} \in \mathbf{T}(j+1)$.
		Therefore, \eqref{eq:1126} holds for any $j\geq 0$ by the induction principle.
		
		In order to show \eqref{eq:112603}, we also use the induction argument.
		For $j=1$, there is only one $\mathcal{T}$ in $\mathbf{T}(1)$ consisting of the root node, its five children and the five edges linking to its five children.
		This implies
		$$\Psi_{\phi}(\mathcal{T})(t)=\mathcal{I}[S(\tau)\phi,S(\tau)\phi,S(\tau)\phi,S(\tau)\phi,S(\tau)\phi](t).$$
		We can use consecutively Minkowski's inequality, Young's inequality, H\"{o}lder's inequality 
		and {the unitarity} of $S(t)$ on $FL^1$ and $FL^2$ to compute
		{\begin{align} \label{eq:1127}
				\|\Psi_{\phi}(\mathcal{T})(t)\|_{FL^\infty}
				\lesssim& \int_0^t\|\widehat{(S(\tau)\phi)}\ast\widehat{(S(\tau)\phi)}\|_{L^\infty}
				\|\hat{\phi}\ast\hat{\phi}\ast\hat{\phi}\|_{L^1}d\tau \notag\\
				\lesssim& t\|\phi\|^2_{L^2}\|\phi\|_{FL^1}^3,
		\end{align}} 
		which verifies \eqref{eq:112603} for $j=1$.
		
		For general $j\geq 1$, assuming that \eqref{eq:112603} holds for all positive integers $k\leq j$.
		For any $\mathcal{T}\in\mathbf{T}(j+1)$,
		as in the previous arguments,
		we get the subtrees $\{\mathcal{T}_{l},l=1,\dots,5\}$ of $\mathcal{T}$, such that $\mathcal{T}_l\in\mathbf{T}(j_l)$ with $j_l\leq j$ and $j_1+\cdots +j_5=j$.
		Since $j\geq 1$, there are at least one among $\{j_1,\dots,j_5\}$ that is not smaller than $1$.
		Without loss of generality, we assume $j_1\geq 1$,
		and hence $\mathcal{T}_1$ has at least one non-terminal node.
		Then, using Minkowski's inequality and Young's inequality we get
		\begin{equation}
			\|\Psi_{\phi}(\mathcal{T})(t)\|_{FL^\infty}
			\lesssim \int_0^t \|\Psi_\phi(\mathcal{T}_1)(\tau)\|_{FL^\infty}
			{\prod_{l=2}^5 
				\|\Psi_\phi(\mathcal{T}_l)(\tau)\|_{FL^1}d\tau.}
		\end{equation}	
		This along with the induction hypothesis (since $j_1\leq j$) and \eqref{eq:1126} implies that
		\begin{align}
			\|\Psi_{\phi}(\mathcal{T})(t)\|_{FL^\infty}
			\lesssim& 
			\int_0^t \tau^{j_1}\|\phi\|^2_{L^2}\|\phi\|^{4j_1-1}_{FL^1}
			{\prod_{l=2}^4\left(\tau^{j_l}\|\phi\|^{4j_l+1}_{FL^1}\right)d\tau}  \notag \\
			\lesssim& \int_0^t \tau^{j} \|\phi\|^2_{L^2}\|\phi\|^{4(j_1+\cdots+j_5)+3}_{FL^1}d\tau \notag \\
			\lesssim& t^{j+1}\|\phi\|^2_{L^2}\|\phi\|^{4(j+1)-1}_{FL^1}.
		\end{align}
		This verifies \eqref{eq:112603} for $\mathcal{T} \in \mathbf{T}(j+1)$ and hence \eqref{eq:112603} is valid for any $j\geq 1$ by the induction principle.
		This completes the proof.
	\end{proof}
	
	\begin{lemma}\label{lem:11270}
		Given any $1\leq p\leq \infty$, 	
		there exists some positive constant $C$ such that
		\begin{equation} \label{Thetauphi-Thetaphi}
			\left\|\Theta_j(u_0+\phi)(t)-\Theta_j(\phi)(t)\right\|_{FL^p}
			\leq C^jt^j\|u_0\|_{FL^p}\left(\|u_0\|^{4j}_{FL^1} + \|\phi\|^{4j}_{FL^1} \right)
		\end{equation}
		holds for all $u_0\in FL^p\cap FL^1$, 
		$\phi \in FL^1$ and $j\in \mathbb{N}$.
	\end{lemma}
	\begin{proof}
		By the multi-linearity of $\Psi_{\phi}$, we can get
		\begin{align}\label{eq:1111}
			\Theta_j(u_0+\phi)(t)-\Theta_j(\phi)(t)
			&= \sum_{\mathcal{T}\in\mathbf{T}(j)}\left(\Psi_{u_0+\phi}(\mathcal{T})-\Psi_{\phi}(\mathcal{T})\right)\nonumber\\	
			&= \sum_{\mathcal{T}\in\mathbf{T}(j)}\sum_{\phi_1,\dots,\phi_{4j+1}}\Psi(\mathcal{T};\phi_1,\dots,\phi_{4j+1})	
		\end{align}
		where the second summation runs over {all the choices of $\phi_1,\dots,\phi_{4j+1}\in\{u_0,\phi\}$ in which at least one $u_0$ appears.}
		By a simple combinatorical argument, we first get
		\begin{equation}\label{eq:11270}
			\sum_{\phi_1,\dots,\phi_{4j+1}} 1\leq 2^{4j+1}.
		\end{equation}
		
		We next use ideas in the proof of \eqref{eq:1126} and \eqref{PsiphiT-FL1-esti} to bound
		\begin{equation}\label{eq:000}
			\|\Psi(\mathcal{T};\phi_1,\dots,\phi_{4j+1})\|_{FL^p}
			\lesssim t^j \|u_0\|_{FL^p}\left(\|u_0\|^{4j}_{FL^1}+ \|\phi\|^{4j}_{FL^1} \right).
		\end{equation}
		
		Taking into account Lemma \ref{lem:1127} $(i)$,
		we can substitute \eqref{eq:11270} and \eqref{eq:000} into \eqref{eq:1111} to obtain \eqref{Thetauphi-Thetaphi}. This finishes the proof.
	\end{proof}
	
	By virtue of Lemmas \ref{Lem-Multiesti-Theta} and \ref{lem:11270},
	we can prove Proposition \ref{prop:lwp} by using the contraction mapping principle.
	The details are omitted here.
	
	Next, we shall prove Theorem \ref{thm:main} for $s$ satisfying \eqref{s<0}. 
	More precisely, we show that 
	
	\begin{proposition}\label{prop:norm:inflation}
		Let $d$ and $s$ be as in \eqref{s<0}. Let $u_0$ be any Schwartz function. Then,
		there exist a family $\{u_n\}_{n\geq 1}$ of local solutions to \eqref{eq:main}
		and a sequence $(t_n)_{n\geq1}$ of real numbers descending to $0$ as $n\rightarrow\infty$,
		such that
		\begin{equation} \label{un0-u0-Hs}
			\left\|u_n(0)-u_0\right\|_{H^s}<\frac{1}{n},
		\end{equation}
		and
		\begin{equation}  \label{un-tn-Hs}
			\|u_n(t_n)\|_{H^s}>n.
		\end{equation}
	\end{proposition} 
	
	In order to prove Proposition \ref{prop:norm:inflation},
	we choose the functions $\phi_n$, $n\geq 1$, such that
	\begin{equation}\label{eq:112707}
		\hat{\phi}_n(\xi)=R \left(1_{N\mathbf{e}_1+Q_A}(\xi)+1_{2N\mathbf{e}_1+Q_A}(\xi)\right),
	\end{equation}
	where $Q_A:=[-A/2,A/2)^d$,  $R=R(n)>0$, $A=A(n)$
	$N=N(n)$,
	and $\mathbf{e}_1=(1,0,\cdots,0) \in \mathbb{R}^d$. These parameters {will be chosen properly} (see Claim \ref{claim:3}) so that
	\begin{equation}\label{eq:11280}
		RA^d\gg \|u_0\|_{FL^1}\sim 1, 
		\ \ \ 1\ll A\ll N.
	\end{equation}
	We can compute directly to find that
	\begin{equation}\label{par:N:A}
		\|\phi_n\|_{H^s}\sim R A^{\frac{d}{2}}N^s,\ \  \|\phi_n\|_{FL^1}\sim R A^{
			d}.
	\end{equation}
	
	By Proposition \ref{prop:lwp},
	there exists a solution $u_n\in C([0,T]; FL^1)$ to \eqref{eq:4nls},
	starting from 
	{$$u_{0,n}:=u_0+\phi_n$$} 
	where 
	{\begin{equation} \label{T-RAd}
			T\sim \left(\|u_0+\phi_n\|_{FL^1}\right)^{-4}\sim \left(RA^d\right)^{-4}.
	\end{equation}} 
	What's more, the solution $u_n(t)$ is indeed given by a series expansion
	\begin{equation*}
		u_n(t)=\sum_{j=0}^\infty\Theta_j(u_0+\phi_n)(t),
	\end{equation*}
	which is uniformly convergent in $FL^1$ with respect to $t\in [0,T]$.

	\begin{lemma}\label{lem:key}
		Let $d$ and $s$ be as in \eqref{s<0}. Let $\phi_n$ and $u_{0,n}$ be as above. For $t\in(0,T]$, we have
		\begin{align}	
			\|u_{0,n}-u_0\|_{H^s}\lesssim&\  RA^{\frac{d}{2}}N^s.\label{eq:112701}\\
			\left\|\Theta_0(u_{0,n})\right\|_{H^s}\lesssim&\  1+RA^{\frac{d}{2}}N^s.\label{eq:112702}\\
			\left\|\Theta_1(u_{0,n})(t)-\Theta_1(\phi_n)(t)\right\|_{H^s}\lesssim&\  t\|u_0\|_{L^2}\left(1+\|\phi_n\|^4_{FL^1}\right)\sim t\|u_0\|_{L^2}R^4A^{4d}.\label{eq:112703}
		\end{align}
		Moreover, for any $j\in \mathbb{N}_+$,
		\begin{equation}\label{error:j}
			\left\|\Theta_j(u_{0,n})(T)\right\|_{H^s}\leq C^j T^j(RA^d)^{4j}\left(Rf_s(A)+\|u_0\|_{L^2}\right).
		\end{equation}
		Here {$C>0$ is a universal constant}
		and
		\begin{equation}\label{eq:f}
			f_s(A):=\left\{
			\begin{split}
				&1\ \ \ \ \ \ \ \ \ \ \ \ \ \  \ \, \mathrm{if}\ s<-\frac{d}{2},\\
				&\left(\log A\right)^{\frac{1}{2}}\ \ \ \mathrm{if}\ s=-\frac{d}{2},\\
				&A^{\frac{d}{2}+s}\ \ \ \ \ \ \ \ \ \ \mathrm{if}\ \ s>-\frac{d}{2}.
			\end{split}
			\right.
		\end{equation}
		In particular, for $T$ small enough so that $TR^4A^{4d}\ll 1$,
		we have
		\begin{equation}\label{eq:112704}
			\left\|\sum_{j=2}^\infty\Theta_j(u_{0,n})(T)\right\|_{H^s}\lesssim T^2R^9A^{8d}f_s(A).
		\end{equation}
	\end{lemma}
	
	\begin{proof}
		Since $u_{0,n}-u_0=\phi_n$, estimate \eqref{eq:112701} follows immediately from \eqref{par:N:A}.
		Since $\Theta_0(u_{0,n})(t)=S(t)(u_{0,n})$,
		estimate \eqref{eq:112702} follows from the triangle inequality and \eqref{par:N:A}. {Although the estimate \eqref{eq:112703} can be proved by a general argument (see the forthcoming \eqref{Thetaj-diff}), it can be shown directly, as is done now.} Since $s<0$, we have
		\begin{equation}
			\left\|\Theta_1(u_{0,n})(t)-\Theta_1(\phi_n)(t)\right\|_{H^s}\leq \left\|\Theta_1(u_{0,n})(t)-\Theta_1(\phi_n)(t)\right\|_{L^2}=\left\|\Theta_1(u_{0,n})(t)-\Theta_1(\phi_n)(t)\right\|_{FL^2}.
		\end{equation}
		Then, estimate \eqref{eq:112703} follows from Lemma \ref{lem:11270} and estimate \eqref{par:N:A}.
		
		We turn to prove \eqref{error:j}. Note that,  $\mathrm{supp}\,\hat{\phi}_n$ is the union of two disjoint cubes of volume $\sim A^d$.
		Moreover, for each $\mathcal{T}\in\mathbf{T}(j)$, the distribution $\Psi_{\phi_n}(\mathcal{T})$ is the $j$-fold iterated integration in the time variable
		of the $4j+1$ products of $S(t)\phi_n$ and their conjugates.
		Therefore, $\mathrm{supp}\, \mathcal{F}[\Psi_{\phi_n}(\mathcal{T})]$ is contained in the union of $2^{4j+1}$ cubes of volume $\sim A^d$.
		This implies that for each $\mathcal{T}\in\mathbf{T}(j)$,
		\begin{equation*}
			\left|\mathrm{supp}\, \mathcal{F}[\Psi_{\phi_n}(\mathcal{T})]\right|\leq c\left|C^jQ_A\right|.
		\end{equation*} 
		{Taking into account \eqref{eq:112601} and Lemma \ref{lem:1127} $(i)$ we get 
			\begin{align*}
				\left|\mathrm{supp}\, \mathcal{F}(\Theta_j(\phi_n))\right|\leq c\left|C^jQ_A\right|, 
			\end{align*}  
			for some other constants, which we still denote by $c$ and $C$.} 
		
		Since $s<0$, the map $\mathbb{R}^d\ni\xi\longmapsto\langle\xi\rangle^s$ is a decreasing function in $|\xi|$.
		It follows that
		\begin{align}\label{eq:090701}
			\left\|\langle\xi\rangle^s\right\|_{L^2_\xi(\mathrm{supp}\mathcal{F}(\Theta_j(\phi_n)))} \leq&\left\|\langle\xi\rangle^s\right\|_{L^2_\xi(cC^jQ_A)}  \leq C^jf_s(A), 
		\end{align} 
		where $f_s(A)$ is given by \eqref{eq:f}. 
		Taking into account \eqref{eq:112602} and \eqref{par:N:A}
		we lead to
		\begin{align} \label{Thetaj-phin-fA}
			\left\|\Theta_j(\phi_n)(t)\right\|_{H^s}
			\leq& \left\|\langle\xi\rangle^s\right\|_{L^2_\xi(\mathrm{supp}\mathcal{F}(\Theta_j(\phi_n)))}\ \left\|\Theta_j(\phi_n)\right\|_{FL^\infty} \notag \\
			\lesssim& C^jt^j\left(RA^d\right)^{4j}Rf_s(A).
		\end{align}
		
		Moreover, since $s<0$,
		we can use Lemma \ref{lem:11270} with $p=2$, \eqref{eq:11280} {and \eqref{par:N:A}} to compute 
		\begin{align} \label{Thetaj-diff}
			\left\|\Theta_j(u_{0,n})(t)-\Theta_j(\phi_n)(t)\right\|_{H^s}
			\leq&\left\|\Theta_j(u_0+\phi_n)(t)-\Theta_j(\phi_n)(t)\right\|_{L^2} \notag \\
			\leq&C^jt^j\|u_0\|_{L^2}\left(\|u_0\|^{4j}_{FL^1}+\|\phi_n\|^{4j}_{FL^1}\right)  \notag \\
			\sim& C^jt^j\|u_0\|_{L^2}\left(RA^d\right)^{4j}.
		\end{align}
		Note that this also proves \eqref{eq:112703}, by taking $j=1$.
		Thus, combining \eqref{Thetaj-phin-fA} and \eqref{Thetaj-diff}
		we obtain \eqref{error:j}.  Since $T$ is small such that $TR^4A^{4d}\ll1$ and $\|u_0\|_{L^2}\lesssim f_s(A)$,
		the last estimate \eqref{eq:112704} follows by adding \eqref{error:j} from $j=2$ to $\infty$.
	\end{proof}
	
	In view of Lemma \ref{lem:key},
	we see that the sum of the Picard iterations of orders
	larger than two are uniformly bounded.
	On the contrary,
	the forthcoming Lemma \ref{prop:lower:bound} shows that
	the first Picard iteration is bounded from below. 
	
	{	\begin{lemma}\label{prop:lower:bound}
			Let $s<0$ and $\phi_n$ be as in \eqref{eq:112707}.
			Then, for $0<t\ll N^{-4}$, we have
			\begin{equation}\label{eq:esti:key}
				\left\|\Theta_1(\phi_n)(t)\right\|_{H^s}\geq tR^5A^{4d}f_s(A).
			\end{equation}
	\end{lemma}}

	\begin{proof} 
		The main tool in the proof is the following technical result, concerning the lower bound of the convolution of characteristic functions.
		\begin{lemma}  \label{lem:112705}
			There exists some positive constant $c$ such that
			for any  $a_k\in{\mathbb{R}^d}$, $1\leq k\leq 5$,
			\begin{equation*}
				cA^{4d}1_{(a_1+a_2+a_3+a_4+a_5)+Q_A}(\xi)\leq 1_{a_1+Q_A}\ast1_{a_2+Q_A}\ast1_{a_3+Q_A}\ast1_{a_4+Q_A}\ast 1_{a_5+Q_A}(\xi).
			\end{equation*}
		\end{lemma}
		\noindent The proof of this lemma can be done by iterating \cite[Lemma 3.5]{oh2017}, which is omitted here.
		
		Let's continue the proof. We have on the Fourier side
		\begin{align}
			\mathcal{F}\left[\Theta_1(\phi_n)(t)\right](\xi)
			=&{i \mu e^{it|\xi|^4}
				\int_{\xi=\xi_1-\xi_2+\xi_3-\xi_4+\xi_5}\int_0^te^{-i\tau(|\xi|^4-|\xi_1|^4+|\xi|_2^4
					-|\xi_3|^4+|\xi_4|^4-|\xi_5|^4)}d\tau }  
			\notag \\
			&\qquad \qquad  \qquad \qquad \qquad \times\hat{\phi}_n(\xi_1)\bar{\hat{\phi}}_n(\xi_2)\hat{\phi}_n(\xi_3)\bar{\hat{\phi}}_n(\xi_4)\hat{\phi}_n(\xi_5)
			d\xi_1d\xi_2d\xi_3d\xi_4d\xi_5.
		\end{align}
		
		Note that,
		for any fixed $1\leq j\leq 5$, if $\xi_j\in\mathrm{supp}\hat{\phi}_n$,
		then {$|\xi_j|\lesssim N$}.
		Taking into account $|\tau|\ll N^{-4}$
		we get that for $\xi=\xi_1-\xi_2+\xi_3-\xi_4+\xi_5$,
		{\begin{equation*}
				\left|\tau(|\xi|^4-|\xi_1|^4+|\xi_2|^4-|\xi_3|^4+|\xi_4|^4-|\xi_5|^4)\right|\ll 1.
		\end{equation*}}
		It follows that
		{\begin{equation*}
				\Re \int_0^t 
				e^{-i\tau(|\xi|^4-|\xi_1|^4+|\xi_2|^4-|\xi_3|^4+|\xi_4|^4-|\xi_5|^4)}d\tau\geq \frac{1}{2}t.
		\end{equation*}} 
		{Therefore, in view of Lemma \ref{lem:112705},}  
		we have the lower bound
		\begin{equation*}
			\left|\mathcal{F}\left[\Theta_1(\phi_n)(t)\right](\xi)\right|\geq tR^5A^{4d}1_{Q_A}(\xi).
		\end{equation*}
		This, together with the fact that  $\|\langle\xi\rangle^s\|_{L^2_\xi(Q_A)}\sim f_s(A)$
		yields \eqref{eq:esti:key}, finishing the proof.
	\end{proof}
	
	{We are now ready to prove Proposition \ref{prop:norm:inflation}. } 
	
	\begin{proof}[Proof of Proposition \ref{prop:norm:inflation}.]
		Using the series expansion in \eqref{eq:power:series}, 
		Lemmas \ref{lem:key} and \ref{prop:lower:bound},
		we can derive that
		\begin{align}\label{eq:090702}
			\left\|u_n(t_n)\right\|_{H^s}
			=&\left\|\sum_{j=0}^\infty\Theta_j(u_{0,n})(t_n) \right\|_{H^s}\\
			\geq& \left\|\Theta_1(\phi_n)(t_n)\right\|_{H^s}-\left\|\Theta_1(u_{0,n})(t_n)-\Theta_1(\phi_n)(t_n) \right\|_{H^s}\nonumber\\
			&\ \qquad \qquad  \qquad -\left\|\Theta_0(u_{0,n})(t_n)\right\|_{H^s}-\left\|\sum_{j=2}^\infty\Theta_j(u_{0,n})(t_n)\right\|_{H^s}\nonumber\\
			\geq& t_nR^5A^{4d}f_s(A)-t_nR^4A^{4d}-RA^{\frac{d}{2}}N^s-t_n^2R^9A^{8d}f_s(A).\nonumber
		\end{align}
		\begin{claim}\label{claim:3}
			We can choose 
			the parameters $t_n,R,A,N$ such that
			\begin{align*}
				&(i)\ \ \  RA^{\frac{d}{2}}N^s\ll \frac{1}{n}, \\
				& (ii)\ \ \  t_nR^4A^{4d}\ll 1, \\
				& (iii)\ \  t_nR^5A^{4d}f_s(A)\gg t_n^2R^9A^{8d}f_s(A),\\
				& (iv)\ \  t_nR^5A^{4d}f_s(A)\gg n,\\
				&  (v)\ \ \  t_n\ll N^{-4},\\
				& (vi)\ \ \ \eqref{eq:11280} \ \mathrm{holds}.
			\end{align*}
		\end{claim}
		From Claim \ref{claim:3} it follows that \eqref{un-tn-Hs} holds, i.e., 
		\begin{equation}
			\left\|u_n(t_n)\right\|_{H^s}\gg n.
		\end{equation}
		Moreover, by \eqref{eq:112701} and $(i)$ in Claim \ref{claim:3}, we have
		\begin{equation}
			\|u_{0,n}-u_0\|_{H^s}\ll \frac{1}{n},
		\end{equation}
		which verifies \eqref{un0-u0-Hs}.
		Up to now, we have proved 
		Proposition \ref{prop:norm:inflation} {modulo the proof of Claim \ref{claim:3}}, to which we now turn.
		
		\begin{proof}[Proof of Claim \ref{claim:3}]
			Note that $(ii)$ implies $(iii)$.
			Hence, it suffices to show $(i)$, $(ii)$, $(iv)$, $(v)$ and $(vi)$.
			Below we consider three cases
			$s<-\frac d2$ when $d\geq 1$,
			$s=-\frac d2$ when $d\geq 1$,
			and $-\frac d2 <s<0$ when $d\geq 2$,
			separately.
			
			\vspace{1mm}
			\noindent\underline{Case $s<-\frac{d}{2}$ when $d\geq 1$.} In this case, we choose
			\begin{equation}
				A=N^{\frac{1}{d}\left(1-\delta\right)},\ \ R=N^{2\delta}, \ \ t_n=N^{-4-5\delta}, \ \ \mathrm{and}\ \ 
				{N:= 10\max\{n^{\frac 1\delta}, n^{-\frac{2}{2s+1 + \delta}}\}, } 
			\end{equation}
			where $\delta>0$ is sufficiently small so that $s<-\frac{1}{2}-\frac{3}{2}\delta$.
			
			Then $N\gg1$, from which it follows that $(v)$ holds. Since $\delta >0$, we have $RA^d=N^{1+\delta}\gg 1$, verifying $(vi)$. Taking $n$ sufficiently large we have
			\begin{equation}
				RA^{\frac{d}{2}}N^s=N^{s+\frac{1}{2}+\frac{3}{2}\delta}\ll\frac{1}{n},\ \  t_nR^4A^{4d}=N^{-\delta}\ll\frac{1}{n},\  \ t_nR^5A^{4d}=N^{\delta}\gg n,
			\end{equation}
			which verify $(i),(ii)$ and $(iv)$, respectively.
			
			\vspace{1mm}
			\noindent\underline{Case $s=-\frac{d}{2}$ when $d\geq 1$.} In this case, we take
			\begin{equation}
				A=\frac{N^{\frac{1}{d}}}{(\log N)^{\frac{1}{16d}}},\ \ R=1,\ \  t_n=\frac{1}{N^4(\log N)^{\frac{1}{8}}},\ \ 
				\mathrm{and}\ \ 
				{ N := e^{100 n^{100}}.} 
			\end{equation}
			As in the first case, we have $N\gg 1$ and hence $(v)$ holds trivially. For such a choice of $N$, we also have $RA^d=N^{1}(\log N)^{-\frac{1}{16}}\gg 1$, verifying $(vi)$. 
			{By simple calculations, we find that}
			\begin{equation}
				RA^{\frac{d}{2}}N^s=N^{\frac{1}{2}(1-d)}\left(\log N\right)^{-\frac{1}{32}}\ll \frac{1}{n},
				\ t_nR^4A^{4d}=\frac{1}{(\log N)^{\frac{3}{8}}}\ll 1,
				\ t_nR^5A^{4d}(\log A)^{\frac{1}{2}}\sim (\log N)^{\frac{1}{8}}\gg n,
			\end{equation}
			which verify the conditions $(i),(ii)$ and $(iv)$, respectively.
			
			\vspace{1mm}
			\noindent\underline{Case $-\frac{d}{2}<s<0$ when $d\geq 2$.}
			In this case, we take
			\begin{equation}
				A=N^{\frac{2}{d}-\delta},\ \  R=N^{-1-s+\frac{d}{2}\delta-\theta},\ \  t_n=N^{-4+4s+3\theta+2d\delta}, \ \ 
				\mathrm{and}\ \ 
				{N:= n^{-\frac{10}{2\theta +\delta s}}}, 
			\end{equation}
			{where $1\gg \delta\gg\theta>0$} are small parameters such that
			\begin{equation}
				-s\delta >2\theta\ \ \mathrm{and}\ \ -4s>3\theta+2d\delta.
			\end{equation}
			Since $3\theta+2d\delta+4s<0$, we have $N\gg 1$ and hence $(v)$ holds trivially. Since both $\delta$ and $\theta$ are small, we have $RA^d=N^{1-s-\frac{d}{2}\delta-\theta}\gg 1$, verifying $(vi)$.
			Taking $n$ sufficiently large  we have
			\begin{equation}
				RA^{\frac{d}{2}}N^s=N^{-\theta}\ll \frac{1}{n},
				\ t_nR^4A^{4d}=N^{-\theta}\ll 1,
				\ t_nR^5A^{4d}A^{\frac{d}{2}+s}=N^{-s\left(1-\frac{2}{d}\right)+\left(-2\theta-s\delta\right)}\gg n
			\end{equation}
			which verify the conditions $(i),(ii)$ and $(iv)$, respectively.
			
		\end{proof}
		Therefore, we complete the proof of Proposition \ref{prop:norm:inflation}.  
	\end{proof}

	\section{The case $0<s<s_{cr}(d)$ for $3\leq d\leq 5$} \label{Sec-s-posit}
	
	In this section,
	we are going to prove
	Theorem \ref{thm:main22}.
	More precisely, we will show
	\begin{proposition}\label{prop:un}
		Let $3\leq d\leq 5$ and $0<s<s_{cr}(d)$.
		Then, for any Schwartz function $u_0$,
		there exists a family  $\{u_n\}_{n\geq 1}$ of solutions to \eqref{eq:main} with $\mu=1$, 
		such that for any $T>0$, 
		\begin{equation}
			\|u_n(0)-u_0\|_{H^s}\rightarrow 0\ \ \mathrm{and}\ \ \|u_n\|_{C([0,T],H^s)}\rightarrow \infty,\ \ as\ n\to \infty. 
		\end{equation}
	\end{proposition}
	
	Quite differently from the previous case where $s<0$, 
	the proof of Proposition \ref{prop:un} relies crucially on the ODE approach in \cite{burqtzvetkov2008,xia2021}. 
	
	To be precise, 
	let us first note that
	$V(t):=e^{i t}$  solves the ODE
	\begin{equation}\label{eq:ode}
		\left\{
		\begin{split}
			&i\frac{d}{dt}V +  |V|^4V=0, \\
			&V(0)=1.
		\end{split}
		\right.
	\end{equation}
	In particular, $V$ is $2\pi$-periodic.
	Moreover,
	straightforward computations show that
	\begin{equation} \label{vn}
		v_n(t) := \psi_n  V \left(t|\psi_n|^4 \right)
	\end{equation}
	solves the equation
	\begin{equation} \label{eq:ode:n}
		\left\{
		\begin{split}
			&i {\partial_t}v_n + |v_n|^4v_n =0,\\
			&v_n(0)=\psi_n.
		\end{split}
		\right.
	\end{equation}
	Here $\{\psi_n\}$ is a family of smooth functions, which we are going to specify.
	
	Setting
	\begin{equation}\label{eq:para}
		q_1:=\frac{d}{2}-s\ \ \mathrm{and}\ \ \kappa_n:=\left(\log\log n\right)^{-\delta_1}
	\end{equation}
	for some $\delta_1>0$ to be specified later, we take
	\begin{equation} \label{psin}
		\psi_n(x):=\kappa_n n^{q_1}\phi(nx),
	\end{equation}
	where $\phi$ is any given compactly supported non-zero function. Then, we have
	\begin{equation}\label{eq:091503}
		\|\psi_n\|_{H^s}\sim\kappa_n\|\phi\|_{H^s},
	\end{equation}
	which tends to zero as $n$ approaches infinity by the choice of $\kappa_n$ in \eqref{eq:para}. 
	
	Substituting \eqref{psin} into \eqref{vn} we obtain
	\begin{equation} \label{vn-V}
		v_n(t) =\kappa_n n^{q_1}\phi(nx)V\left(t\big(\kappa_n n^{q_1}\phi(nx)\big)^4\right).
	\end{equation}
	It then follows from \cite[Lemma A.3]{burqtzvetkov2008} that
	\begin{equation}\label{eq:091504}
		\|v_n(t_n)\|_{H^s}\geq \kappa_n  \left[t_n\big(\kappa_n n^{q_1}\big)^4\right]^s \to \infty,
	\end{equation}
	where $t_n$ is chosen appropriately of form
	\begin{equation}\label{time}
		t_n:=\big(\kappa_n n^{q_1}\big)^{-4}(\log n)^{\delta_2} 
	\end{equation}
	for some small $\delta_2>0$ to be specified later.

	We next compare $v_n$ to the space-time function $u_n$,
	which solves equation \eqref{eq:main} with the initial data $u_n(0):= u_0+ \psi_n$,
	namely,
	\begin{equation}\label{eq:main:n}
		\left\{
		\begin{split}
			& i\partial_tu_n+\Delta^2 u_n +  |u_n|^4u_n=0,   \\
			& u_n(0) =u_0+\psi_n.
		\end{split}
		\right.
	\end{equation}
	{		Let us digress a bit to specify the lifespan of $u_n$.
		Thanks to $3\leq d\leq 5$, the Cauchy problem \eqref{eq:main:n} falls into the energy sub-critical regime.  Since $u_0+\psi_n$ is a Schwartz function for each $n$, we can use the energy estimates to construct local solutions in energy space via the Picard iteration.  Noticing that the equation is defocusing, we can then use the conservation laws of mass \eqref{eq:mass} and energy \eqref{eq:energy} to extend this local solution to a global one. See \cite{dinh2018b} and \cite{ghanmisaanouni2016} for details. Thus the choice of $t_n$ is allowable for each $\delta_2>0$.}
	
	To continue,
	we introduce the semiclassical energy defined by
	\begin{equation} \label{En-energy}
		E_n[u]:=n^{s}\|u\|_{L^2}+n^{s-4}\|\Delta^2u\|_{L^2}.
	\end{equation}
	We use $E_n[\cdot]$ to measure the size of {the remainder}
	\begin{equation} \label{wn-def}
		w_n:=u_n-u^L-v_n,
	\end{equation}
	where $u_n$, $v_n$ are given as above
	and {$u^L:=e^{it\Delta^2}u_0$} is the linear wave.

	It follows from equations \eqref{eq:ode:n}, \eqref{eq:main:n} and \eqref{wn-def} that
	$w_n$ satisfies the equation
	\begin{equation}\label{eq:main:n:v}
		\left\{
		\begin{split}
			& i\partial_tw_n+\Delta^2 w_n=-\Delta^2v_n+ |v_n|^4v_n-|v_n+u^L+w_n|^4(v_n+u^L+w_n)=:-\Delta^2v_n+G_n,\\
			& w_n(0)=0,
		\end{split}
		\right.
	\end{equation}
	where
	\begin{align} \label{Gn-def}
		G_n :=  |v_n|^4v_n-|v_n+u^L+w_n|^4(v_n+u^L+w_n).
	\end{align}
	
	Using the basic energy inequality for the equation of Schr\"odinger type we get
	\begin{equation}\label{eq:energy:w:n}
		\left|\frac{d}{dt}E_n[w_n](t)\right|\leq n^{s}\|-\Delta^2v_n+G_n\|_{L^2}+n^{s-4}\|-\Delta^4v_n+\Delta^2G_n\|_{L^2}
	\end{equation}
	
	In order to control the right-hand-side of \eqref{eq:energy:w:n}, we first give the following lemma.
	
	\begin{lemma} (Control of $v_n$)   \label{lem:con:vn}
		For every $1\leq j\leq 4$, every multi-index $\alpha\in\mathbb{N}^d$ 
		and any $p\in [2,\infty]$, 
		we have that for any $t\in [0,t_n]$,  
		\begin{equation} \label{p-vnj-L2}
			\left\|\partial^\alpha \left(v^j_n\right)\right\|_{L^p}
			\lesssim \left(\kappa_n n^{q_1}\right)^{j}\left(\log n\right)^{|\alpha|\delta_2}n^{|\alpha|-\frac{d}{p}}.
		\end{equation}
		As a consequence, for any $t\in [0,t_n]$, we have
		\begin{align}
			\|\Delta^2v_n(t,\cdot)\|_{L^2} 
			\lesssim& \kappa_nn^{q_1}\left(\log n\right)^{4\delta_2}n^{4-\frac{d}{2}}, \label{Delta2vn} \\
			\|\Delta^4v_n(t,\cdot)\|_{L^2} 
			\lesssim& \kappa_nn^{q_1}\left(\log n\right)^{8\delta_2}n^{8-\frac{d}{2}}.  \label{Delta4vn} 
		\end{align} 
		{Here all the implicit constants are independent of $n$.} 
	\end{lemma}
	
	\begin{proof}
		Let $\beta$ be a multi-index that is arbitrarily given. Using \eqref{vn-V} 
		{and \eqref{time} we can find some positive spatial Schwartz function $v^\beta$ such that 
			\begin{align}
				\left|\partial^\beta v_n(x) \right| \lesssim \kappa_nn^{q_1} 
				\left(\log n\right)^{|\beta|\delta_2}n^{|\beta|}
				v^\beta(nx), 
				\label{eq:12230}
			\end{align}
			holds for all $t\in[0,t_n]$. Taking $L^p$-norm on both sides, we obtain 
			\begin{align}
				\left\|\partial^\beta v_n(t,\cdot) \right\|_{L^p} &\lesssim \kappa_nn^{q_1} \left(\log n\right)^{|\beta|\delta_2}n^{|\beta|-\frac{d}{p}}.\label{eq:12231}
			\end{align}
			By proper linear combinations, from these we can derive \eqref{Delta2vn}  and  \eqref{Delta4vn}. } 
		
		Let $\alpha$ be given in the statement. We use Leibniz's rule to expand
		{\begin{align}
				\left\|\partial^\alpha \left(v^j_n\right)\right\|_{L^p}
				&\lesssim  \sum_{\substack{\alpha_1+\cdots+ \alpha_j=\alpha\\ |\alpha_1| \geq\cdots\geq |\alpha_j|}}
				\left\|\partial^{\alpha_1}v_n  \cdots\partial^{\alpha_j}v_n\right\|_{L^p}. \notag  
			\end{align}
			Using H\"{o}lder's inequality and \eqref{eq:12231}, we can further bound
			\begin{align}
				\left\|\partial^\alpha \left(v^j_n\right)\right\|_{L^p}&\leq \sum_{\substack{\alpha_1+\cdots+ \alpha_j=\alpha\\ |\alpha_1| \geq\cdots\geq |\alpha_j|}}  \|\partial^{\alpha_1}v_n\|_{L^p} \prod_{l=2}^{j}\left\|\partial^{\alpha_l}v_n\right\|_{L^\infty} \notag  \\
				&\lesssim \sum_{\substack{\alpha_1+\cdots+ \alpha_j=\alpha\\ |\alpha_1| \geq\cdots\geq |\alpha_j|}}
				\left(\kappa_nn^{q_1}(\log n)^{|\alpha_1|\delta_2}n^{|\alpha_1|-\frac{d}{p}}\right) \prod_{l=2}^{j}
				\left(\kappa_nn^{q_1}\left(\log n\right)^{|\alpha_l|\delta_2}n^{|\alpha_l|}\right) \notag  \\
				&\lesssim \left(\kappa_nn^{q_1}\right)^j(\log n)^{|\alpha|\delta_2}n^{|\alpha|-\frac{d}{p}}.\nonumber
		\end{align}} 
		This proves \eqref{p-vnj-L2} and hence completes the proof.
	\end{proof}
	
	Putting \eqref{eq:energy:w:n}, \eqref{Delta2vn} and \eqref{Delta4vn} together we obtain
	\begin{equation}\label{eq:energy:w:n:r}
		\left|\frac{d}{dt}E_n[w_n](t)\right| 
		\lesssim \kappa_nn^{q_1}\left(\log n\right)^{8\delta_2}n^{4+s-\frac{d}{2}}+ n^{s}\|G_n\|_{L^2}+n^{s-4}\|\Delta^2G_n\|_{L^2}
	\end{equation}
	
	We turn to control $\|G_n\|_{L^2}$ and 
	$\|\Delta^2 G_n\|_{L^2}$, as is done 
	in Propositions \ref{Lem-Gn-L2} and \ref{Lem-D2Gn-L2} below. 
	Before stating the main results, let us first present the estimate 
	for the linear wave which will be used frequently later.

	\begin{lemma} (Control of $u^L$) \label{Lem-uL-L9-pnu}
		For any $p\in[2,\infty]$, any positive integer $j$ and any multi-index $\alpha\in\mathbb{N}^d$, we have that for all $t\geq 0$, 
		\begin{align*} 
			\left\|\partial^\alpha \left[\left(u^L(t)\right)^j\right] \right\|_{L^{p}} \lesssim 1, 
		\end{align*} 
		where the implicit constant depends on $\alpha,j$ and $p$. 
	\end{lemma}
	
	\begin{proof}Observe that by using Leibniz's rule and induction argument, it suffices to show
		\begin{align} \label{uL-L9-pnu}
			\left\|\partial^\alpha u^L(t) \right\|_{L^{p}} \lesssim_{p,\alpha} 1.
		\end{align} 
		Since $u_0$ is a Schwartz function, by the Sobolev embedding $H^d \hookrightarrow L^{p}$ and the unitarity of $e^{it\Delta^2}$ on $H^d$, we have that 
		\begin{align} 
			\|\partial^{\alpha}u^L\|_{L^{p}}
			\lesssim \|\partial^{\alpha}u^L\|_{H^d}
			= \|\partial^{\alpha}u_0\|_{H^d}
			\lesssim 1. 
		\end{align}
		This gives \eqref{uL-L9-pnu} and finishes the proof of Lemma \ref{Lem-uL-L9-pnu}.
	\end{proof}
	
	We first control $\|G_n\|_{L^2}$ as follows.
	\begin{proposition} (Estimate of $\|G_n\|_{L^2}$) \label{Lem-Gn-L2}  We have that for all $n\geq 1$ and $t\in[0,t_n]$, 
		\begin{equation}  \label{Gn-L2-En}
			n^{s}\|G_n\|_{L^2}\lesssim n^{s}(1+\left(\kappa_nn^{q_1}\right)^4n^{-\frac{d}{2}})
			+(\kappa_nn^{q_1})^4 E_n[w_n]+ n^{-4s+2d} E_n[w_n]^5.
		\end{equation}
	\end{proposition}
	
	\begin{proof}
		It follows from \eqref{Gn-def} that
		\begin{equation}\label{eq:2100:*}
			|G_n|\lesssim |u^L|^5+ |w_n|^5+|w_n||v_n|^4+|u^L||v_n|^4. 
		\end{equation}
		Taking $L^2$-norm on both sides and using Minkowski's inequality, we obtain
		\begin{equation} \label{Gn-L2}
			\|G_n\|_{L^2}\lesssim \|u^L\|^5_{L^{10}}+\|w_n\|^5_{L^{10}}
			+\|u^Lv^4_n\|_{L^2}+\|w_nv^4_n\|_{L^2}.
		\end{equation}
		In the following, we control the right hand side of this inequality term by term.
		
		For the first term, we apply Lemma \ref{Lem-uL-L9-pnu} with $j=1$ to find 
		\begin{align} \label{uL-L9}
			\|u^L\|^5_{L^{10}}
			\lesssim 1.
		\end{align}
		
		For the second term $\|w_n\|^5_{L^{10}}$, since $d\leq 5$, we first use the Gagliardo-Nirenberg inequality and then \eqref{En-energy} to obtain
		\begin{eqnarray*}
			\|w_n\|_{L^{10}}
			\lesssim   \|w_n\|^{1-\frac{d}{10}}_{L^2}\|\Delta^2w_n\|^{\frac{d}{10}}_{L^2}
			\leq n^{-s}n^{\frac{2d}{5}}E_n[w_n].
		\end{eqnarray*}
		Consequently we have
		\begin{equation} \label{wn-L10}
			\|w_n\|^5_{L^{10}}\lesssim n^{-5s+2d} E_n[w_n]^5.
		\end{equation}
		
		For the third term $\|u^Lv^4_n\|_{L^2}$, we first use H\"{o}lder's inequality to bound
		{\begin{eqnarray} \label{eq:12234}
				\|u^Lv^4_n\|_{L^2}
				\leq \|u^L\|_{L^{\infty}} \|v_n^4\|_{L^2}.
		\end{eqnarray}} 
		On the one hand, we use Lemma \ref{Lem-uL-L9-pnu} with $p=\infty,j=1$ and $|\alpha|=0$ to obtain
		\begin{equation}\label{eq:12232}
			\|u^L\|_{L^\infty}\lesssim 1.
		\end{equation}
		On the other hand, we apply \eqref{p-vnj-L2} with $p=2, |\alpha|=0$ and $j=4$ to get 
		{\begin{equation}\label{eq:12233}
				\|v_n^4\|_{L^2}
				\lesssim \left(\kappa_n n^{q_1}\right)^{4}n^{-\frac{d}{2}}.
		\end{equation} } 
		Inserting \eqref{eq:12232} and \eqref{eq:12233} into \eqref{eq:12234}, we obtain
		\begin{equation}\label{uLvn4-L2}
			\|u^Lv^4_n\|_{L^2}\lesssim \left(\kappa_n n^{q_1}\right)^{4}n^{-\frac{d}{2}}
		\end{equation}

		For the last term $\|w_nv^4_n\|_{L^2}$, since $d\leq 5$, we use Gagliardo-Nirenberg's inequality, \eqref{En-energy} and Lemma \ref{lem:con:vn} to compute 
		\begin{align} \label{wnvn4-L2}
			\|w_nv^4_n\|_{L^2}
			\lesssim& \|w_n\|_{L^4} \|v^4_n\|_{L^4}   \nonumber \\
			\lesssim& \|w_n\|^{1-\frac{d}{4}\left(\frac{1}{2}-\frac{1}{4}\right)}_{L^2}\|\Delta^2w_n\|^{\frac{d}{4}\left(\frac{1}{2}-\frac{1}{4}\right)}_{L^2}
			\|v^4_n\|_{L^4} \nonumber \\
			\lesssim&  (\kappa_nn^{q_1})^4n^{-s}E_n[w_n].
		\end{align}
		
		{Finally, plugging \eqref{uL-L9}, \eqref{wn-L10}, \eqref{uLvn4-L2} and \eqref{wnvn4-L2} into \eqref{Gn-L2}}, we obtain
		\eqref{Gn-L2-En}. This completes the proof.
	\end{proof}
	
	{We turn to control $\|\Delta^2 G_n\|_{L^2}$
		on the right-hand-side of \eqref{eq:energy:w:n:r}}. 
	\begin{proposition} (Estimate of $\|\Delta^2 G_n\|_{L^2}$)   \label{Lem-D2Gn-L2}
		For all $n\geq 1$ and $t\in[0,t_n]$, we have
		\begin{align} \label{eq:2206}
			n^{s-4}\|\Delta^2G_n\|_{L^2}
			\lesssim& n^{s-4}\left(1+\left(\kappa_n n^{q_1}\right)^{4}\left(\log n\right)^{4\delta_2}n^{4-\frac{d}{2}}\right)  \notag \\ 
			&+ \left(\log n\right)^{4\delta_2}  \sum\limits_{j=1}^5
			\left(\kappa_n n^{q_1}\right)^{5-j}
			n^{(j-1)(-s+\frac d2)} E_n[w_n]^j.
		\end{align}
	\end{proposition}
	
	In order to prove Proposition \ref{Lem-D2Gn-L2}, 
	{we will not only use Lemmas \ref{lem:con:vn} and \ref{Lem-uL-L9-pnu}} to control $v_n$ and the linear wave, but also will prove the following three Lemmas, {which are used to control  $w_n$, 
		the interactions between $u^L$ and  $w_n$, 
		and the interactions between $v_n$ and $w_n$.  
	}

	\begin{lemma} (Control of $w_n$)   \label{lem:dwn}
		Fix $1\leq j\leq 5$ and a multi-index $\alpha\in\mathbb{N}^d$ with $0\leq |\alpha| \leq 4$. Then for any $n$ and any $t\in[0,t_n]$ we have
		\begin{equation}  \label{p-wnj-L2}
			\left\|\partial^\alpha\left(w_n^j\right)\right\|_{L^2}\lesssim n^{|\alpha|-js+\frac{d}{2}(j-1)} E_n[w_n]^j.
		\end{equation}
	\end{lemma}
	
	\begin{proof}The case $j=1$ can be proved directly by referring to \eqref{En-energy}. For $j\geq 2$, by Leibniz's rule, {we get}
		\begin{align} \label{eq:0}
			\left\|\partial^\alpha\left(w_n^j\right)\right\|_{L^2}
			&\lesssim \sum_{\substack{\alpha_1+\cdots+ \alpha_j=\alpha  \\ 
					|\alpha_1| \geq\cdots\geq |\alpha_j|}}
			\left\|\partial^{\alpha_1}w_n\cdots\partial^{\alpha_j}w_n\right\|_{L^2} 
		\end{align}
		We now estimate the right-hand-side 
		in the following four different cases. 
		
		\noindent\underline{$(i)$ The case $|\alpha_1|=|\alpha|$}. 
		We use H\"{o}lder's inequality to bound
		\begin{align}\label{eq:12260}
			\left\|\partial^{\alpha_1}w_n\cdots\partial^{\alpha_j}w_n\right\|_{L^2} \leq \left\|\partial^{\alpha_1}w_n\right\|_{L^2} \left\|w_n\right\|^{j-1}_{L^\infty}.
		\end{align}
		Using \eqref{En-energy} and interpolation arguments 
		we bound 
		\begin{equation}\label{eq:122601}
			\left\|\partial^\alpha w_n\right\|_{L^2}
			\lesssim n^{-s+|\alpha|}E_n[w_n].
		\end{equation}
		Moreover, since $d\leq 5$, we use Gagliardo-Nirenberg's inequality to obtain
		\begin{equation}\label{eq:122604}
			\left\|w_n\right\|_{L^\infty}\lesssim \|w_n\|_{L^2}^{1-\frac{d}{8}}
			\|\Delta^2 w_n\|_{L^2}^{\frac{d}{8}}\lesssim n^{-s}n^{\frac{d}{2}}E_n[w_n], 
		\end{equation}
		which implies that 
		\begin{equation}\label{eq:122602}
			\left\|w_n\right\|^{j-1}_{L^\infty}\lesssim \left(n^{-s}n^{\frac{d}{2}}E_n[w_n]\right)^{j-1}. 
		\end{equation}
		Inserting \eqref{eq:122601} and \eqref{eq:122602} into \eqref{eq:12260}, we have
		\begin{equation}\label{eq:1}
			\left\|\partial^{\alpha_1}w_n\cdots\partial^{\alpha_j}w_n\right\|_{L^2}\lesssim	n^{|\alpha|-js+\frac{d}{2}(j-1)} E_n[w_n]^j. 
		\end{equation} 
		
		\noindent\underline{$(ii)$ The case $|\alpha_1|=|\alpha|-1$ and $|\alpha_2|=1$}. Using H\"{o}lder's inequality, we get
		\begin{equation}\label{eq:12261}
			\left\|\partial^{\alpha_1}w_n\cdots\partial^{\alpha_j}w_n\right\|_{L^2}\leq \left\|\partial^{\alpha_1}w_n\right\|_{L^p} \left\|\partial^{\alpha_2}w_n\right\|_{L^{\tilde{p}}}\left\|w_n\right\|^{j-2}_{L^\infty},
		\end{equation}
		where the positive numbers $p$ and $\tilde{p}$ satisfy
		\begin{equation*}
			\frac{1}{p} = \frac{1}{2}-\frac{1}{d}\ \mathrm{and}\ \ \frac{1}{\tilde{p}} = \frac{1}{d}.
		\end{equation*}
		For the first factor on the right-hand-side of \eqref{eq:12261}, we use Gargliardo-Nirenberg's inequality to bound
		\begin{equation}\label{eq:122610}
			\left\|\partial^{\alpha_1}w_n\right\|_{L^p}\lesssim \|w_n\|^{1-\frac{d}{4}\left(\frac{1}{2}-\frac{1}{p}+\frac{|\alpha_1|}{d}\right)}_{L^2}\|\Delta^2w_n\|^{\frac{d}{4}\left(\frac{1}{2}-\frac{1}{p}+\frac{|\alpha_1|}{d}\right)}_{L^2}=\left\|w_n\right\|_{L^2}^{1-\frac{|\alpha|}{4}}\left\|\Delta^2w_n\right\|_{L^2}^{\frac{|\alpha|}{4}} 
			\lesssim n^{{|\alpha|}-s}E_n[w_n].
		\end{equation}
		By the same argument, we bound the second factor on the right-hand-side of \eqref{eq:12261} as follows
		\begin{equation}\label{eq:122611}
			\left\|\partial^{\alpha_2}w_n\right\|_{L^{\tilde{p}}}\lesssim \|w_n\|^{1-\frac{d}{8}}_{L^2}\|\Delta^2w_n\|^{\frac{d}{8}}_{L^2}\lesssim n^{\frac{d}{2}-s}E_n[w_n].
		\end{equation}
		Similarly to \eqref{eq:122602}, it follows from \eqref{eq:122604}
		\begin{equation}\label{eq:122612}
			\left\|w_n\right\|^{j-2}_{L^\infty}
			\lesssim \left(n^{-s}n^{\frac{d}{2}}E_n[w_n]\right)^{j-2}. 
		\end{equation}
		Inserting \eqref{eq:122610}, \eqref{eq:122611} and \eqref{eq:122612} into \eqref{eq:12261}, we obtain
		\begin{equation}\label{eq:2}
			\left\|\partial^{\alpha_1}w_n\cdots\partial^{\alpha_j}w_n\right\|_{L^2}\lesssim	n^{|\alpha|-js+\frac{d}{2}(j-1)} E_n[w_n]^j.
		\end{equation}	
		
		\noindent\underline{$(iii)$ The case $|\alpha_1|=|\alpha_2|=|\alpha_3|=1$}. 
		We first treat the sub-case $|\alpha|=3$. Notice that this happens only for $j\geq 3$. Hence we use H\"{o}lder's inequality to compute
		\begin{equation}\label{eq:12270:}
			\left\|\partial^{\alpha_1}w_n\cdots\partial^{\alpha_j}w_n\right\|_{L^2} \leq \left(\prod_{l=1}^3\left\|\partial^{\alpha_l}w_n\right\|_{L^6}\right)  \left\|w_n\right\|^{j-3}_{L^\infty}.
		\end{equation} 
		For each $l\in\{1,2,3\}$, since $d\leq 5$, we can use Gagliardo-Nirenberg's inequality to bound
		\begin{equation*}
			\left\|\partial^{\alpha_l}w_n\right\|_{L^6}
			\lesssim \left\|w_n\right\|_{L^2}^{1-\frac{d}{4}\left(\frac{1}{3}+\frac{1}{d}\right)}\left\|\Delta^2w_n\right\|^{\frac{d}{4}\left(\frac{1}{3}+\frac{1}{d}\right)}_{L^2}
			\lesssim n^{d\left(\frac{1}{3}+\frac{1}{d}\right)-s}E_n[w_n].
		\end{equation*}
		As a consequence, we have
		\begin{equation}\label{eq:122701}
			\left(\prod_{l=1}^3\left\|\partial^{\alpha_l}w_n\right\|_{L^6}\right) \lesssim n^{d+3-3s}E_n[w_n]^3.
		\end{equation} 
		This along with \eqref{eq:122604} tells that
		\begin{equation}\label{eq:3}
			\left\|\partial^{\alpha_1}w_n\cdots\partial^{\alpha_j}w_n\right\|_{L^2}\lesssim n^{|\alpha|-js+\frac{d}{2}(j-1)} E_n[w_n]^j.
		\end{equation}
		
		We next treat the sub-case $|\alpha|=4$. Note that this happens only for $j\geq 4$. Thus, we use H\"{o}lder's inequality to bound
		\begin{equation}\label{eq:12270}
			\left\|\partial^{\alpha_1}w_n\cdots\partial^{\alpha_j}w_n\right\|_{L^2} \leq \left(\prod_{l=1}^4\left\|\partial^{\alpha_l}w_n\right\|_{L^8}\right)  \left\|w_n\right\|^{j-4}_{L^\infty}.
		\end{equation}
		We can argue similarly as in the previous sub-case to see that \eqref{eq:3} still holds for $|\alpha|=4$. 
		
		\noindent\underline{$(iv)$ The case $|\alpha_1|=|\alpha_2|=2$}. Note that this happens only for $j\geq 4$.
		We use H\"{o}lder's inequality to bound
		\begin{equation}\label{eq:12271}
			\left\|\partial^{\alpha_1}w_n\cdots\partial^{\alpha_j}w_n\right\|_{L^2}\leq \left\|\partial^{\alpha_1}w_n\right\|_{L^4} \left\|\partial^{\alpha_2}w_n\right\|_{L^{4}}\left\|w_n\right\|^{j-2}_{L^\infty}.
		\end{equation}
		For $l\in \{1, 2\}$, since $d\leq 5$, we use Gagliardo-Nirenberg's inequality to obtain
		\begin{equation}\label{eq:kk}
			\left\|\partial^{\alpha_l}w_n\right\|_{L^{4}}\lesssim \|w_n\|_{L^2}^{1-\left(\frac{d}{16}+\frac{1}{2}\right)} \|\Delta^2w_n\|_{L^2}^{\frac{d}{16}+\frac{1}{2}}\lesssim n^{\frac{d}{4}+2-s}E_n[w_n]. 
		\end{equation}
		As a consequence, we have
		\begin{equation}\label{eq:122710}
			\left\|\partial^{\alpha_1}w_n\right\|_{L^4} \left\|\partial^{\alpha_2}w_n\right\|_{L^{4}} \lesssim n^{\frac{d}{2}+4-2s}E_n[w_n]^2. 
		\end{equation}
		Plugging \eqref{eq:122710} and \eqref{eq:122612} into \eqref{eq:12271} we get
		\begin{equation}\label{eq:4}
			\left\|\partial^{\alpha_1}w_n\cdots\partial^{\alpha_j}w_n\right\|_{L^2}\lesssim	n^{|\alpha|-js+\frac{d}{2}(j-1)} E_n[w_n]^j.
		\end{equation}
		
		\noindent\underline{$(v)$ The case $|\alpha_1|=2$ and $|\alpha_2|=|\alpha_3|=1$}. This happens only for $j\geq 4$. We 
		use H\"{o}lder's inequality to bound
		\begin{equation}\label{eq:12272}
			\left\|\partial^{\alpha_1}w_n\cdots\partial^{\alpha_j}w_n\right\|_{L^2}\leq \left\|\partial^{\alpha_1}w_n\right\|_{L^4} \left\|\partial^{\alpha_2}w_n\right\|_{L^{8}}\left\|\partial^{\alpha_3}w_n\right\|_{L^{8}}\left\|w_n\right\|^{j-3}_{L^\infty}.
		\end{equation}
		For $l\in\{2,3\}$, since $d\leq 5$, we use Gagliardo-Nirenberg's inequality to bound
		\begin{equation*}
			\|\partial^{\alpha_l}w_n\|_{L^8}\lesssim \|w_n\|_{L^2}^{1-\frac{d}{4}\left(\frac{3}{8}+\frac{1}{d}\right)}\|\Delta^2w_n\|_{L^2}^{\frac{d}{4}\left(\frac{3}{8}+\frac{1}{d}\right)}\lesssim n^{d\left(\frac{3}{8}+\frac{1}{d}\right)-s}E_n[w_n]. 
		\end{equation*}
		Consequently, we have
		\begin{equation}\label{eq:122721}
			\left\|\partial^{\alpha_2}w_n\right\|_{L^{8}}\left\|\partial^{\alpha_3}w_n\right\|_{L^{8}}\lesssim n^{2d\left(\frac{3}{8}+\frac{1}{d}\right)-2s}E_n[w_n]^2 
		\end{equation}
		Note that we can use \eqref{eq:kk} to bound the first factor in \eqref{eq:12272} and \eqref{eq:122604} to bound the last factor on the right hand \eqref{eq:12272}. Thus inserting \eqref{eq:kk}, \eqref{eq:122604} and \eqref{eq:122721} into \eqref{eq:12272}, we obtain
		\begin{equation}\label{eq:5}
			\left\|\partial^{\alpha_1}w_n\cdots\partial^{\alpha_j}w_n\right\|_{L^2}\lesssim	n^{|\alpha|-js+\frac{d}{2}(j-1)} E_n[w_n]^j.
		\end{equation}
		
		Note that, $(i)$ and $(ii)$ cover all the cases $|\alpha|=2$, $(i), (ii)$ and $(iii)$ cover all the cases $|\alpha|=3$, and $(i), (ii), (iii), (iv)$ and $(v)$ cover all the cases $|\alpha|=4$. Therefore, collecting \eqref{eq:1}, \eqref{eq:2}, \eqref{eq:3}, \eqref{eq:4} and \eqref{eq:5} altogether  
		we deduce from \eqref{eq:0} that  \eqref{p-wnj-L2} holds. 
		This finishes the proof of Lemma \ref{lem:dwn}.
	\end{proof}

	\begin{lemma} (Interaction between $u^L$ and $w_n$) \label{lem:dwnuL}
		For every $1\leq j , l \leq 4$
		and any multi-index $|\alpha|\leq 4$, we have
		\begin{equation} \label{p-uL-wn-L2}
			\left\|\partial^\alpha\left(\left(u^L\right)^jw_n^l\right)\right\|_{L^2}
			\lesssim n^{|\alpha| -ls+ \frac{d}{2}(l-1)}E_n[w_n]^{l}.
		\end{equation}
	\end{lemma}
	
	\begin{proof}
		Using consecutively Leibniz's rule and H\"older's inequality,  {we get}
		\begin{align}
			\left\|\partial^\alpha \left(\left(u^L\right)^jw_n^l\right)\right\|_{L^2}
			\lesssim&  \sum_{\alpha_1+\alpha_2=\alpha}
			\left\|\partial^{\alpha_1} \left((u^L)^j\right) \partial^{\alpha_2} \left(w^l_n\right)\right\|_{L^2} \notag \\
			\lesssim&  \sum_{\alpha_1+\alpha_2=\alpha}
			\|\partial^{\alpha_1} \left((u^L)^j\right) \|_{L^{\infty}}
			\| \partial^{\alpha_2} \left(w^l_n\right) \|_{L^2}.
		\end{align}
		We are now in a position to apply Lemmas \ref{Lem-uL-L9-pnu} and \ref{lem:dwn} to obtain
		\begin{align*}
			\left\|\partial^\alpha \left(\left(u^L\right)^jw_n^l\right)\right\|_{L^2}
			\lesssim \sum_{|\alpha_2|\leq |\alpha|} \left\|\partial^{\alpha_2}\left(w^l_n\right)\right\|_{L^2}
			\lesssim  \sum_{|\alpha_2|\leq |\alpha|} n^{|\alpha_2| -ls+\frac{d}{2}(l-1)}E_n[w_n]^{l}
			\lesssim n^{|\alpha| -ls + \frac{d}{2}(l-1)}E_n[w_n]^{l},
		\end{align*}
		which yields \eqref{p-uL-wn-L2}.
	\end{proof} 
	
	\begin{lemma} (Interaction between $v_n$ and $w_n$) \label{lem:dvnwn}
		For every $1\leq j, l\leq 4$
		and any multi-index $|\alpha|\leq 4$
		we have
		\begin{equation}   \label{p-vnj-wnl-L2} 
			\left\|\partial^\alpha \left(v_n^jw_n^l\right)\right\|_{L^2}
			\lesssim \left(\kappa_nn^{q_1}\right)^j\left(\log n\right)^{|\alpha|\delta_2}n^{|\alpha| -ls + \frac{d}{2}(l-1)}E_n[w_n]^l.
		\end{equation}
	\end{lemma}
	
	\begin{proof}
		Using consecutively Leibniz's rule, H\"{o}lder's inequality, Lemmas \ref{lem:con:vn} and \ref{lem:dwn}, we compute
		\begin{align*}
			\left\|\partial^\alpha \left(v_n^jw_n^l\right)\right\|_{L^2}
			&\lesssim \sum_{\alpha_1+\alpha_2=\alpha}
			\left\|\partial^{\alpha_1} \left(v^j_n\right) \partial^{\alpha_2} \left(w^l_n\right)\right\|_{L^2}\\
			&\lesssim \sum_{\alpha_1+\alpha_2=\alpha}
			{	\left\|\partial^{\alpha_1} \left(v^j_n\right)\right\|_{L^\infty}}\left\|\partial^{\alpha_2} \left(w^l_n\right)\right\|_{L^2}\\
			&\lesssim \sum_{\alpha_1+\alpha_2=\alpha}
			\left(\kappa_n n^{q_1}\right)^{j}\left(\log n\right)^{|\alpha_1|\delta_2}n^{|\alpha_1|} n^{|\alpha_2| -ls + \frac{d}{2}(l-1)}E_n[w_n]^{l}\\
			&\lesssim \left(\kappa_n n^{q_1}\right)^{j} 
			\left(\log n\right)^{|\alpha|\delta_2}n^{|\alpha| -ls + \frac{d}{2}(l-1)}E_n[w_n]^{l}.
		\end{align*} 
	\end{proof}

	We are now ready to prove Proposition \ref{Lem-D2Gn-L2}.
	\begin{proof}[Proof of Proposition \ref{Lem-D2Gn-L2}] 
		
		We first use Leibniz's rule 
		and  Minkowski's inequality to compute
		\begin{align}\label{eq:2101}
			n^{s-4}\|\Delta^2G_n\|_{L^2} 
			\lesssim& n^{s-4}\sum_{j=0}^4\left\|\Delta^2(v_n^j(u^L)^{5-j})\right\|_{L^2} + n^{s-4}\sum_{j=0}^4\left\|\Delta^2\left(v_n^j\left(u^L\right)^{4-j}w_n\right)\right\|_{L^2} \notag \\    &+n^{s-4}\sum_{j=0}^3\left\|\Delta^2\left(v_n^j\left(u^L\right)^{3-j}w_n^2\right)\right\|_{L^2}+n^{s-4}\sum_{j=0}^2\left\|\Delta^2\left(v_n^j\left(u^L\right)^{2-j}w_n^3\right)\right\|_{L^2}\nonumber\\  
			&+n^{s-4}\sum_{j=0}^1\left\|\Delta^2\left(v_n^j\left(u^L\right)^{1-j}w_n^4\right)\right\|_{L^2}+n^{s-4}\left\|\Delta^2\left(w^5_n\right)\right\|_{L^2}\\
			&=:I+II+III+IV+V+VI\nonumber. 
		\end{align}
		
		We will control terms $I,II,\dots,VI$ one by one.
		
		\noindent\underline{To bound $I$}.
		For $j=0$, by Lemma \ref{Lem-uL-L9-pnu} we have
		\begin{equation}\label{eq:2102}
			\left\|\Delta^2\left((u^L)^5\right)\right\|_{L^2}\lesssim 1. 
		\end{equation} 
		To continue, we use Lemmas \ref{lem:con:vn} and \ref{Lem-uL-L9-pnu} to compute
		\begin{align}\label{eq:121}
			\sum_{j=1}^4\left\|\Delta^2\left(v_n^j(u^L)^{5-j}\right)\right\|_{L^2}
			&\lesssim \sum_{j=1}^4\sum_{|\alpha_1|+|\alpha_2|= 4}
			\left\|\partial^{\alpha_1} \left(v^j_n\right)\  \partial^{\alpha_2}\left((u^L)^{5-j}\right)\right\|_{L^2}\nonumber\\
			&\lesssim \sum_{j=1}^4\sum_{|\alpha_1|+|\alpha_2|= 4}
			\left\|\partial^{\alpha_1} \left(v^j_n\right) \right\|_{L^2}
			\left\|\partial^{\alpha_2} \left((u^L)^{5-j}\right)\right\|_{L^\infty}\nonumber\\
			&\lesssim  \sum_{j=1}^4\sum_{|\alpha_1|+|\alpha_2|= 4}
			\left(\kappa_n n^{q_1}\right)^{j}\left(\log n\right)^{|\alpha_1|\delta_2}n^{|\alpha_1|-\frac{d}{2}}\\
			&\lesssim  \sum_{j=1}^4\left(\kappa_n n^{q_1}\right)^{j}\left(\log n\right)^{4\delta_2}n^{4-\frac{d}{2}}\nonumber.
		\end{align}
		Putting \eqref{eq:2102} and \eqref{eq:121} together and multiplying the resulted inequality by $n^{s-4}$, {we get}
		\begin{equation}\label{eq:2200}
			n^{s-4}\sum_{j=0}^4\left\|\Delta^2(v_n^j(u^L)^{5-j})\right\|_{L^2}\lesssim n^{s-4}\left(1+\left(\kappa_n n^{q_1}\right)^{4}\left(\log n\right)^{4\delta_2}n^{4-\frac{d}{2}}\right). 
		\end{equation}
		
		\noindent\underline{To bound $II$}.
		For the case $j=0$, we apply Lemma \ref{lem:dwnuL} with $j=4$, $l=1$, $|\alpha|=4$ to see
		\begin{equation}\label{eq:1212}
			\left\|\Delta^2\left((u^L)^4w_n\right)\right\|_{L^2}
			\lesssim n^{4-s}E_n[w_n].
		\end{equation}
		Next, we use Lemmas \ref{Lem-uL-L9-pnu} and \ref{lem:dvnwn} to obtain 
		\begin{align}\label{eq:1211}
			\sum_{j=1}^4\left\|\Delta^2\ \left(v_n^j\left(u^L\right)^{4-j}w_n\right)\right\|_{L^2}
			&\lesssim\sum_{j=1}^4 \sum_{|\alpha_1|+|\alpha_2|=4}
			\left\|\partial^{\alpha_1} \left(v_n^jw_n\right)\  \partial^{\alpha_2} \left((u^L)^{4-j}\right)\right\|_{L^2}\nonumber\\
			&\lesssim \sum_{j=1}^4\sum_{|\alpha_1|+|\alpha_2|=4}
			\left\|\partial^{\alpha_1} \left(v_n^jw_n\right)\right\|_{L^2} \left\|\partial^{\alpha_2}\left((u^L)^{4-j}\right)\right\|_{L^\infty}\\
			&\lesssim \sum_{j=1}^4\left(\kappa_nn^{q_1}\right)^j\left(\log n\right)^{4\delta_2}n^{4-s}E_n[w_n]\nonumber.
		\end{align}
		Putting \eqref{eq:1212} and \eqref{eq:1211} together and multiplying the resulted inequality by $n^{4-s}$, we get
		{\begin{equation}\label{eq:2201}
				n^{s-4}\sum_{j=0}^4
				\left\|\Delta^2\left(v_n^j\left(u^L\right)^{4-j}w_n\right)\right\|_{L^2} 
				\lesssim \left(\kappa_nn^{q_1}\right)^4\left(\log n\right)^{4\delta_2}E_n[w_n].
		\end{equation}}
		
		\noindent\underline{To bound $III$}. First, by Lemmas \ref{Lem-uL-L9-pnu} and \ref{lem:dwnuL}
		with $|\alpha|=4$, $j=3$ and $l=2$, we have
		\begin{equation}\label{eq:123}
			\left\|\Delta^2\left((u^L)^3w_n^2\right)\right\|_{L^2}\lesssim n^{4-2s+ \frac{d}{2}}E_n[w_n]^2.
		\end{equation}
		Next, we apply Lemmas \ref{Lem-uL-L9-pnu} and \ref{lem:dvnwn} to compute
		\begin{align}\label{eq:124}
			\sum_{j=1}^3\left\|\Delta^2\left(v_n^j\left(u^L\right)^{3-j}w_n^2\right)\right\|_{L^2}
			&\lesssim \sum_{j=1}^3\sum_{|\alpha_1|+|\alpha_2|=4}
			\left\|\partial^{\alpha_1} \left(v^j_nw^2_n\right) \partial^{\alpha_2} \left((u^L)^{3-j}\right)\right\|_{L^2}\nonumber\\
			&\lesssim \sum_{j=1}^3\sum_{|\alpha_1|\leq 4}
			\left\|\partial^{\alpha_1} \left(v^j_nw^2_n\right)\right\|_{L^2} 
			\left\| \partial^{\alpha_2} \left((u^L)^{3-j}\right) \right\|_{L^\infty}\nonumber \\
			&\lesssim  \sum_{j=1}^3\sum_{|\alpha_1|\leq 4}
			\left(\kappa_n n^{q_1}\right)^j\left(\log n\right)^{|\alpha_1|\delta_2}
			n^{|\alpha_1|-2s + \frac{d}{2}}E_n[w_n]^2\\
			&\lesssim \sum_{j=1}^3\left(\kappa_nn^{q_1}\right)^j\left(\log n\right)^{4\delta_2}n^{4 -2s + \frac{d}{2}}E_n[w_n]^2.\nonumber
		\end{align}
		Then it follows from \eqref{eq:123} and \eqref{eq:124} that
		\begin{equation}\label{eq:2202}
			n^{s-4}\sum_{j=0}^3\left\|\Delta^2\left(v_n^j\left(u^L\right)^{3-j}w_n^2\right)\right\|_{L^2}\lesssim \left(\kappa_nn^{q_1}\right)^3\left(\log n\right)^{4\delta_2}n^{-s+\frac{d}{2}}E_n[w_n]^2.
		\end{equation}
		
		\noindent\underline{To bound $IV$}.
		For $j=0$, by Lemma \ref{lem:dwnuL}, we have  
		\begin{equation}\label{eq:125}
			\left\|\Delta^2\left((u^L)^2w_n^3\right)\right\|_{L^2}\lesssim n^{4-3s+d}E_n[w_n]^3.
		\end{equation}
		{To continue, we use Lemmas \ref{Lem-uL-L9-pnu} and \ref{lem:dvnwn} to compute
			\begin{align}\label{eq:126}
				\sum_{j=1}^2\left\|\Delta^2\left(v_n^j\left(u^L\right)^{2-j}w_n^3\right)\right\|_{L^2}
				&\lesssim \sum_{j=1}^2\sum_{|\alpha_1|+|\alpha_2|=4}
				\left\|\partial^{\alpha_1} \left(v^j_nw^3_n\right) \partial^{\alpha_2}  \left((u^L)^{2-j}\right)\right\|_{L^2} \notag \\
				&\lesssim \sum_{j=1}^2\sum_{|\alpha_1|\leq 4}
				\left\|\partial^{\alpha_1} \left(v^j_nw^3_n\right)\right\|_{L^2} 
				\|\partial^{\alpha_2}  \left((u^L)^{2-j}\right)\|_{L^\infty}\notag  \\
				&\lesssim  \sum_{j=1}^2\sum_{|\alpha_1|\leq 4}
				\left(\kappa_nn^{q_1}\right)^j\left(\log n\right)^{|\alpha_1|\delta_2}
				n^{|\alpha_1|-3s+d}E_n[w_n]^3  \nonumber  \\
				&\lesssim \sum_{j=1}^2\left(\kappa_nn^{q_1}\right)^j
				\left(\log n\right)^{4\delta_2} n^{4-3s+d}E_n[w_n]^3. 
		\end{align}} 
		It then follows from \eqref{eq:125} and \eqref{eq:126} that
		\begin{equation}\label{eq:2203}
			n^{s-4}\sum_{j=0}^2\left\|\Delta^2\left(v_n^j\left(u^L\right)^{2-j}w_n^3\right)\right\|_{L^2}\lesssim \left(\kappa_nn^{q_1}\right)^2\left(\log n\right)^{4\delta_2}n^{-2s+d}E_n[w_n]^3.
		\end{equation}
		
		\noindent\underline{To bound $V$}.
		For $j=0$, we get from Lemma \ref{lem:dwnuL} that
		\begin{equation}\label{eq:127}
			\left\|\Delta^2\left(u^Lw_n^4\right)\right\|_{L^2}\lesssim n^{4-4s+ \frac{3d}{2}}E_n[w_n]^4.
		\end{equation}
		For $j=1$, we use Lemma \ref{lem:dvnwn} to obtain
		\begin{equation}\label{eq:128}
			\left\|\Delta^2\left(v_nw_n^4\right)\right\|_{L^2} 
			\lesssim \left(\kappa_nn^{q_1}\right)\left(\log n\right)^{4\delta_2}n^{4 -4s+ \frac{3d}{2}}E_n[w_n]^4. 
		\end{equation}
		It then follows \eqref{eq:127} and \eqref{eq:128} that
		\begin{equation}\label{eq:2204}
			n^{s-4}\sum_{j=0}^1\left\|\Delta^2\left(v_n^j\left(u^L\right)^{1-j}w_n^4\right)\right\|_{L^2}\lesssim \left(\kappa_nn^{q_1}\right)\left(\log n\right)^{4\delta_2}n^{-3s+\frac{3d}{2}}E_n[w_n]^4.
		\end{equation}
		
		\noindent\underline{To bound $VI$}.
		It follows from Lemma \ref{lem:dwn} 
		with $j=5$ and $|\alpha|=4$ that
		\begin{equation}
			\left\|\Delta^2\left(w^5_n\right)\right\|_{L^2}\lesssim n^{4 -5s+ 2d}E_n[w_n]^5,
		\end{equation}
		and hence
		\begin{equation}\label{eq:2205}
			n^{s-4}\left\|\Delta^2\left(w^5_n\right)\right\|_{L^2}\lesssim  n^{-4s+2d}E_n[w_n]^5.
		\end{equation}
		
		Therefore,
		plugging estimates \eqref{eq:2200}, \eqref{eq:2201}, \eqref{eq:2202}, \eqref{eq:2203},
		\eqref{eq:2204} and \eqref{eq:2205} into \eqref{eq:2101}  we obtain
		\eqref{eq:2206},
		which completes the proof.	
	\end{proof}
	
	The next result shows that, 
	the semi-classical energy of the remainder is indeed smaller than $\kappa_n$ on the short time interval $[0,t_n]$.

	\begin{proposition}\label{claim:1}
		For $n$ sufficiently large, we have
		\begin{equation*}
			E_{n}[w_n](t)<\kappa_n,\ \ \forall t\in [0,t_n]. 
		\end{equation*}
	\end{proposition}

	\begin{proof} 
		For each $n$, we define the set
		\begin{equation} \label{eq:In}
			I_n:=\left\{t\in[0,t_n]\ \big|\ \sup_{0\leq \tau\leq t}E_n[w_n](\tau)<\kappa_n\right\}.
		\end{equation}
		Since $E_n[w_n](0)=0$, the initial time $0$ is in $I_n$ and hence $I_n$ is nonempty. 
		Since $E_n[w_n](t)$ is continuous in $t\in [0,t_n]$,  so is $\sup_{0\leq \tau\leq t}E_n[w_n(\tau)]$. 
		It follows that $I_n$ is open in $[0,t_n]$. 
		
		Thanks to the fact that the interval $[0,t_n]$ is connected, in order to show that $I_n=[0,t_n]$, it suffices to prove that $I_n$ is closed. 
		For this, we fix $T_n:=\sup_{t\in I_n}t$. 
		Then, by the definition of $I_n$, we have $[0,T_n)\subseteq I_n$.  
		Inserting \eqref{Gn-L2-En} and \eqref{eq:2206} into \eqref{eq:energy:w:n:r} we obtain
		\begin{align} \label{eq:2207}
			\frac{d}{dt}E_n[w_n]
			\lesssim& \kappa_nn^{q_1}\left(\log n\right)^{8\delta_2}n^{4+s-\frac{d}{2}}
			+ n^{s}\left(1+\left(\kappa_n n^{q_1}\right)^{4}\left(\log n\right)^{4\delta_2}n^{-\frac{d}{2}}\right) \notag \\ 
			&+\left(\log n\right)^{4\delta_2} 
			{\sum\limits_{j=1}^5 (\kappa_n n^{q_1})^{5-j} n^{-(j-1)(s-\frac d2)} E_n[w_n]^j.}  
		\end{align}

		Since $0<s<s_{cr}(d)$ and $s_{cr}(d)=\frac{d}{2}-1$, we have
		\begin{equation}
			4<4q_1,\ \  s<4q_1.
		\end{equation}
		As a consequence,
		there exists a small number $\delta >0$, so that the inequality
		\begin{equation}\label{eq:2208}
			\kappa_nn^{q_1}\left(\log n\right)^{8\delta_2}n^{4+s-\frac{d}{2}}+ n^{s}\left(1+\left(\kappa_n n^{q_1}\right)^{4}\left(\log n\right)^{4\delta_2}n^{-\frac{d}{2}}\right)\lesssim \left(\kappa_n n^{q_1}\right)^{4-\delta}\left(\log n\right)^{8\delta_2}
		\end{equation}
		holds for all $n$ sufficiently large.
		Note also that, for any $t\in [0,T_n)\subseteq [0,t_n]$, we have 
		{\begin{align}\label{eq:129}
				n^{-(j-1)(s-\frac d2)} E_n[w_n]^{j-1} 
				= n^{(j-1)q_1} E_n[w_n]^{j-1}  
				\leq  (\kappa_n n^{q_1})^{j-1}.
		\end{align}  } 
		It follows from \eqref{eq:2208} and \eqref{eq:129} that
		\begin{equation}
			\frac{d}{dt}E_n[w_n]\leq  C_1\left(\kappa_n n^{q_1}\right)^{4-\delta}\left(\log n\right)^{8\delta_2}+C_2\left(\kappa_nn^{q_1}\right)^{4}\left(\log n\right)^{4\delta_2}E_n[w_n],\ \ \forall t\in[0,T_n)
		\end{equation}
		holds for some positive universal constants $C_1$ and $C_2$. 
		Using Gronwall's inequality we get
		\begin{equation}
			E_{n}[w_n](T_n)\leq \frac{ C_1\left(\kappa_n n^{q_1}\right)^{4-\delta}\left(\log n\right)^{8\delta_2}}{C_2\left(\kappa_nn^{q_1}\right)^{4}\left(\log n\right)^{4\delta_2}}e^{C_2T_n \left(\kappa_nn^{q_1}\right)^4\left(\log n\right)^{4\delta_2}}\leq\frac{C_1\left(\log n\right)^{4\delta_2}}{C_2\left(\kappa_n n^{q_1}\right)^{\delta}}e^{C_2\left(\log n\right)^{5\delta_2}}
		\end{equation}
		where in the second inequality we have used 
		the choice of $t_n$ in \eqref{time}. {Taking $\delta_2>0$ sufficiently small
			such that 
			\begin{align*}
				C_2 (\log n)^{5\delta_2} < \frac 14 q_1 \delta \log n 
			\end{align*} 
			we arrive at } 
		\begin{equation}
			E_n[w_n](T_n)<\kappa_n
		\end{equation}
		for $n$ sufficiently large.  
		Therefore, it follows from the definition of $I_n$ that $T_n\in I_n$ and so $I_n$ is closed. 
		This completes the proof of Proposition \ref{claim:1}.
	\end{proof}
	
	We are now in position to prove Proposition \ref{prop:un}. 
	
	\begin{proof}[Proof of Proposition \ref{prop:un}] 	
		We choose the family $\{\psi_n\}$ as in \eqref{psin}, with $q_1$, $\kappa_n$ as in \eqref{eq:para} and $\phi$ being a non-zero compactly supported smooth function. 
		{For each $n\geq 1$, as used in the above, $u_n$ is the global $H^2$ solution to \eqref{eq:main:n}. } 
		
		On the one hand, it follows from \eqref{eq:091503} that
		\begin{equation}
			\|u_n(0,\cdot)-u_0\|_{H^s}=\|\psi_n\|_{H^s}\rightarrow_{n\rightarrow\infty}0.
		\end{equation}
		
		On the other hand, we deduce from the definition of $E_n[w_n]$ and Proposition \ref{claim:1} that  
		for all $t\in [0,t_n]$, there hold 
		\begin{align*}
			& \left\|w_n(t)\right\|_{L^2}\leq n^{-s}E_n[w_n](t)\leq n^{-s}\kappa_n,   \\
			& \left\|\Delta^2w_n(t)\right\|_{L^2}\leq n^{-s+4}E_n[w_n](t)\leq n^{4-s}\kappa_n. 
		\end{align*} 
		Interpolating these two inequalities we obtain for all $t\in [0,t_n]$, 
		\begin{equation}\label{eq:091505}
			\left\|w_n(t)\right\|_{H^s}\lesssim \left\|w_n(t)\right\|^{1-\frac{s}{4}}_{L^2}\left\|\Delta^2 w_n(t)\right\|^{\frac{s}{4}}_{L^2}
			\lesssim \kappa_n \ \to 0,\ \ as\ \ n\to \infty. 
		\end{equation}
		This along with \eqref{eq:091504} 
		and Lemma \ref{Lem-uL-L9-pnu} yields that 
		\begin{equation}
			\|u_n\|_{L^\infty([0,T],H^s)}\geq \|v_n(t_n)\|_{H^s}-\|w_n\|_{L^\infty H^s}-\|u^L\|_{H^s} 
			\to \infty, \ \mathrm{as}\ n\to \infty. 
		\end{equation}
		This completes the proof of Proposition \ref{prop:un}. 
	\end{proof}

	\noindent{\bf Acknowledgment.} 
	Deng Zhang  thanks the support by NSFC (No. 11871337)
	and Shanghai Rising-Star Program 21QA1404500. 
	
	\bibliographystyle{plain}
	\bibliography{bib}	
\end{document}